\documentclass[11pt]{article}

\usepackage{amscd,amsfonts,amsmath,amssymb,amsthm}
\usepackage[left=30mm,right=30mm, top=3cm, bottom=3cm]{geometry}
\usepackage[math]{iwona}
\usepackage[T1]{fontenc}
\usepackage[toc,page]{appendix}
\usepackage{array}
\usepackage[english]{babel}
\usepackage{mathtools}

\usepackage[english]{babel}
\usepackage{mathtools}
\usepackage{xspace}
\usepackage{booktabs}
\usepackage{subcaption}
\usepackage{amsthm}

\usepackage{multicol}
\usepackage[utf8]{inputenc} % allow utf-8 input
\usepackage[T1]{fontenc}    % use 8-bit T1 fonts
\usepackage[hidelinks]{hyperref}       % hyperlinks
\usepackage{url}            % simple URL typesetting
\usepackage{booktabs}       % professional-quality tables
\usepackage{amsfonts}       % blackboard math symbols
\usepackage{nicefrac}       % compact symbols for 1/2, etc.
\usepackage{microtype}      % microtypography
\usepackage{cleveref}       % smart cross-referencing
\usepackage{graphicx}
\usepackage{doi}
\usepackage{float}
\usepackage{xcolor}
\usepackage{hyperref}

\usepackage{bigstrut} 
\usepackage{bbold}
\usepackage{calc}
\usepackage{caption}
\usepackage{color}
\usepackage{colortbl}
\usepackage{csquotes}
\usepackage{dsfont}
\usepackage{enumerate}
\usepackage{enumitem}
\usepackage{epsfig}
\usepackage{esint}
\usepackage{hyperref}
\usepackage{etoolbox}
\usepackage{float}
\usepackage{graphpap}
\usepackage{graphics}
\usepackage{graphicx}
\usepackage{latexsym}
\usepackage{lipsum}
\usepackage{lmodern}
\usepackage{marginnote}
\usepackage{mathrsfs}
\usepackage{mathtools}
\usepackage{mathabx}        
\usepackage{multirow}
\usepackage{pdfpages}
\usepackage{pgfplots}
\usepackage{pxfonts}
\usepackage{rotating}
\usepackage{subcaption}
\usepackage{tabularx,booktabs}
\usepackage{algorithm}
\usepackage{algpseudocode}
\usepackage{varwidth}
\usepackage{scalerel}
\newcolumntype{Y}{>{\centering\arraybackslash}X}
\usepackage{tikz}
\usetikzlibrary{shapes}
\usetikzlibrary{shapes.misc}

\numberwithin{equation}{section}

\newtheorem{theorem}{Theorem}[section]

\newtheorem{proposition}[theorem]{Proposition}

\newtheorem{remark}[theorem]{Remark}
\newtheorem{definition}[theorem]{Definition}

\makeatletter
\DeclareRobustCommand\onedot{\futurelet\@let@token\@onedot}
\def\@onedot{\ifx\@let@token.\else.\null\fi\xspace}

\def\eg{e.g\onedot} 
\def\ie{i.e\onedot} 
\def\cf{cf\onedot}

\def\etal{et al\onedot}

\makeatother

\DeclareMathAlphabet\mathbfcal{OMS}{cmsy}{b}{n}

\newcommand{\red}[1] {{\color{red}{{#1}}}}

\newcommand{\notinclude}[1]{}

\IfFileExists{todonotes.sty}{%
    \usepackage{todonotes}
    \usepackage{comment}
    \usepackage{setspace}
    \let\todoavailable1	
}{}%

\providecommand\MR{} % workaround because MR is also used in amsref

\ifdefined\todoavailable
    \makeatletter
    \newcommand{\inlinetodo}[2][]{%
        \@todo[caption={}, inline, #1]{#2}%
    } 
    \newcommand{\missing}[2][]{ \@todo[color=red!40, #1]{Missing: #2} }

    \newcommand{\CS}[2][]{\@todo[inline,color=yellow, #1]{Christoph: #2}}
    \newcommand{\JS}[2][]{\@todo[inline,color=yellow, #1]{Josua: #2}}
    \renewcommand{\MR}[2][]{\@todo[inline,color=yellow, #1]{Martin: #2}}
    \makeatother
\else
    
    \newcommand{\missing}[1]{ {\red{missing: [} #1 \red{]!!!}} }

    \newcommand{\CS}[2][]{\red{Christoph: [} #2 \red{].}}
    \renewcommand{\MR}[2][]{\red{Martin: [} #2 \red{].}}
    \newcommand{\JS}[2][]{\red{Josua: [} #2 \red{].}}
    \newcommand{\inlinetodo}[2][]{{\red{TODO: {#2} }\hfill\\}}%
\fi

\newcommand{\R}{\mathbb{R}}
\newcommand{\N}{\mathbb{N}}

\newcommand{\Z}{\mathbb{Z}}

\newcommand{\Id}{\mathrm{Id}}

\DeclareMathOperator*{\argmin}{arg\,min}

\renewcommand{\d}{\,\mathrm{d}}

\newcommand{\dist}{\mathrm{dist}}
\newcommand{\sdist}{\mathrm{sdist}}

\newcommand{\Hd}{\mathcal{H}}

\ifdefined\laplace
\renewcommand{\laplace}{\Delta}
\else
\newcommand{\laplace}{\Delta}
\fi
\newcommand{\eps}{\varepsilon}
\newcommand{\mc}{{\mathbf{h}}}
\newcommand{\param}{x}

\newcommand{\Per}{\mathcal{P}}

\newcommand{\dwell}{\Psi}

\newcommand{\kernel}{\kappa}
\newcommand{\NNkernel}{K}
\title{A hybrid minimizing movement and  neural network approach to Willmore flow}

\author{
	Martin Rumpf \thanks{Institute for Numerical Simulation, University of Bonn, Endenicher Allee 60, 53115 Bonn, Germany
				(\href{mailto:martin.rumpf@uni-bonn.de}{martin.rumpf@uni-bonn.de}, \href{mailto:christoph.smoch@ins.uni-bonn.de}{christoph.smoch@ins.uni-bonn.de} ).}
	\and 
	Josua Sassen\thanks{Centre Borelli, ENS Paris-Saclay, 4 Avenue des Sciences, 91190 Gif-sur-Yvette, France
				(\href{mailto:josua.sassen@ens-paris-saclay.fr}{josua.sassen@ens-paris-saclay.fr}).}
	\and
	Christoph Smoch\footnotemark[1]
}
\begin{document}
\maketitle
\begin{abstract}
	We present a hybrid method combining a minimizing movement scheme with neural operators for the simulation of phase field-based Willmore flow.
	The minimizing movement component is based on a standard optimization problem on a regular grid, whereas the functional to be minimized involves a neural approximation of mean curvature flow proposed by Bretin \etal \cite{BrDeMa22}.
	Numerical experiments confirm stability for large time step sizes, consistency, and significantly reduced computational cost compared to a traditional finite element method. 
	Moreover, applications demonstrate its effectiveness in surface fairing and reconstructing of damaged shapes. 
	Thus, our approach offers a robust and efficient tool for geometry processing.
\end{abstract}

\maketitle

%%%%%%%%%%%%%%%%%%%%%%%%%%%%%%%%
%%%%%%%%%%%%%%%%%%%%%%%%%%%%%%%%
%%%%%%%%%%%%%%%%%%%%%%%%%%%%%%%%
\section{Introduction}
\label{sec:introduction}

%% willmore flow
% general introduction
%In this paper we present a hybrid method for Willmore flow. 
Willmore flow is the $L^2$-gradient flow of the Willmore energy, which is defined as the surface integral over the squared mean curvature.  
For closed surfaces of genus zero, the Willmore energy allows to quantify how much a surface deviates from being a perfect sphere, with a round sphere having minimal Willmore energy.
Physically, the Willmore energy reflects an approximation of the stored energy in a thin elastic shell with a planar physical reference configuration \cite{Ci00}.
It is also used to model the behavior of cell membranes, which have the tendency to minimize their bending.
To this end, an extension of the Willmore energy, the Helfrich model, is used to describe elastic cell membranes in biology \cite{CaArMa14}. %\MR{new references missing}
Furthermore, in computer graphics and geometry processing, Willmore flow is used for surface smoothing \cite{CrPiSc13}, and surface restoration \cite{ClDiDz04,LeCh23}. %\MR{new references missing}
This motivates the study of Willmore flow, and in particular, the development of numerical schemes.

However, as we will discuss below, Willmore flow is described by a fourth-order PDE, which makes it challenging to devise schemes that stably allow for large time steps and are computationally efficient.
We will show that our proposed hybrid scheme is stable for large time steps and indeed leads to an efficient scheme for Willmore flow in the case of a phase field formulation, which is in particular practical for applications with implicitly described geometries.
To compute Willmore flow, we will consider a minimizing movement time discretization.
Therein, we combine a discretization of the minimizing movement scheme on a regular grid with a neural network approximation of the mean curvature arising in the Willmore energy.

% parametric formulation
This paper deals with phase field models approximating hypersurfaces \(\Gamma\) in the computational domain \(\Omega=(0,1)^d\) and their evolution by Willmore flow.
To motivate the hybrid method, we first recall the parametric formulation of Willmore flow.
We denote by \(\param \colon \Gamma \to \Gamma\) the identity restricted to the surface \(\Gamma\).
Then the \emph{Willmore energy} of $\Gamma$ is given by
\begin{equation}
w[\param] \coloneqq \frac{1}{2} \int_\Gamma\mc^2(\param) \d \Hd^{d-1},
\end{equation}
where \(\mc\) is the mean curvature, which we take to be the sum of the principal curvatures.
%of the image of \(\Gamma\) under \(\param\) (which is for \(\param\) being the identity, of course, \(\Gamma\) itself).
The \emph{Willmore flow} for parametrizations is then the \(L^2\)-gradient flow of this energy, \ie the evolution \(\param \colon [0, T] \times \Gamma \to \Omega\) that fulfills 
\begin{equation}
(\partial_t \param, \vartheta)_{L^2(\Gamma)} = -\partial_\param w[\param](\vartheta) 
\end{equation}
for all test functions \(\vartheta \in C^{\infty}(\Gamma, \R^d)\), where \((\cdot,\cdot)_{L^2(\Gamma)}\) denotes the \(L^2\)-scalar product on the hypersurface \(\Gamma\) and \(\partial_\param w[\param](\vartheta) \) the variation of the Willmore energy in direction \(\vartheta\).

% analysis of Willmore flow and Willmore surfaces
Willmore surfaces, \ie minimizers of the Willmore energy, and Willmore flow have been the subject of intense theoretical study.
Simonett proved in \cite{Si01} the existence of a unique and locally smooth solution of Willmore flow for sufficiently smooth initial surfaces as well as exponential convergence to a sphere for initial surfaces close to a sphere.
Similarly, Dall'Acqua \etal~\cite{DaMuSc23} proved that if the initial datum is a torus of revolution with Willmore energy less than \(16 \pi\) then the Willmore flow converges to the Clifford Torus.
%Polden~\cite{Po95,Po96} analytically investigated Willmore flow of curves and surfaces already in 1996.
Kuwert and Sch\"atzle treated long-time existence and regularity of solutions in co-dimension one in \cite{KuSc01,KuSc02c,KuSc04} and Rivi\`ere~\cite{Ri08} extended these results to arbitrary co-dimension.
In 2014, Marques and Neves~\cite{MaNe14} were able to prove the famous Willmore conjecture, \ie that for every smooth immersed torus in \(\R^3\) the Willmore energy is lower bounded by \(4\pi^2\).

%Numerics
Similarly, also the numerical treatment of Willmore flow for surfaces has garnered significant attention.
%Dziuk \etal \cite{DzKuSc02} investigated the long-time existence of Willmore flow of curves and introduced a corresponding numerical algorithm.
Rusu~\cite{Ru01} introduced a semi-implicit finite element scheme for the computation of the parametric Willmore flow of surfaces, which was applied by Clarenz \etal \cite{ClDiDz04} to surface restoration problems.
Droske and Rumpf \cite{DrRu04} introduced a level set formulation for Willmore flow and a corresponding numerical scheme. 
Deckelnick and Dziuk provided in \cite{DeDz06} a priori error estimates for a spatially discretized but time-continuous finite element scheme on two-dimensional graphs. 
%Furthermore, in \cite{DeDz09}, they introduced a similar semi-discrete scheme and error analysis for the case of curves.
%In \cite{DeSc09}, Deckelnick and Schieweck proved the convergence of a conforming finite element approximation in case of axially symmetric surfaces. 
Alternative finite element schemes for parametric Willmore flow were introduced by Barrett \etal \cite{BaGaNu07} and Dziuk \cite{Dz08}.
In contrast, Bobenko and Schr\"oder \cite{BoSc05} introduced a discrete Willmore energy and flow based on discrete differential geometry.
Concerning phase field models, Du \etal introduced and analyzed a discrete semi-implicit scheme for Willmore flow in 
\cite{DuLiWa04, 
	%DuLiWa06, 
	DuWa07}.
In \cite{BrMaOu15}, Bretin \etal investigated flows for various diffuse approximations of the Willmore energy and its relaxations and introduced corresponding numerical schemes.
If the goal is to minimize Willmore energy, one can also consider gradient flows with respect to other metrics. 
For example, Schumacher \cite{Sc17} analyzed \(H^2\)-gradient flows for the Willmore energy with numerical experiments on triangle meshes. 
% while Perl \cite{Pe16} provided a numerical scheme based on subdivision surfaces.
Soliman \etal \cite{SoChDi21} extended this idea on triangle meshes to incorporate further constraints -- most notably on the conformal class of the surface. 
However, in this paper, the goal is not to minimize the Willmore \emph{energy} but to efficiently simulate Willmore \emph{flow}, \ie the \(L^2\)-gradient flow of the Willmore energy.

To this end,  we will consider a \emph{variational time-discretization} of Willmore flow based on the minimizing movements paradigm \cite{De93,AmGiSa08}. 
In case of parametric Willmore flow and time step size \(\tau > 0\), given a parametrization \(\param^k \colon \Gamma \to \Omega\) approximating the evolution at time \(k\tau\) , the next time step \(\param^{k+1}\) is given as the minimizer of 
\begin{equation}
\label{eq:variational_willmore_step}
e[x^k, x] \coloneqq \lVert x - x^k \rVert_{L^2(\Gamma^k)}^2 + 2\tau\, w[\param].
\end{equation}
% approximation via MCF
To obtain a fully implicit, variational formulation of Willmore flow that is (experimentally) unconditionally stable and allows for large step sizes, Balzani and Rumpf \cite{BaRu13} proposed to approximate the Willmore functional \(w[\param]\) in \autoref{eq:variational_willmore_step} using an approximate mean curvature.
To this end, one denotes by \(y_{\tilde\tau}[\param] \colon \Gamma \to \Omega\) the solution of one discrete time step of mean curvature flow (MCF) with time step size \(\tilde\tau\) starting from the initial parametrization \(\param\) corresponding to the surface $\Gamma$. 
By definition, the mean curvature \(\mc\) is the normal velocity of mean curvature flow. 
Hence, \(\frac12 \int_\Gamma \frac{\vert y_{\tilde \tau}[x]-x\vert^2}{\tilde \tau^2} \d \Hd^{d-1}\) constitutes a difference quotient approximation of the Willmore energy \(w[\param]\) and a minimizing movement functional is given as
\begin{equation}
\label{eq:nested_willmore_step}
e[x^k, x] \coloneqq \Vert x - x^k \Vert_{L^2(\Gamma^k)}^2 + \tau \int_\Gamma \frac{\vert y_{\tilde \tau}[x]-x\vert^2}{\tilde \tau^2} \d \Hd^{d-1}.
\end{equation}
As a discrete solution for mean curvature motion, one might consider the time step of a semi-discrete backward Euler scheme with time step $\tilde \tau$ as proposed by Dziuk \cite{Dz91}. 
Thus, minimizing the energy  $e[x^k, x]$ in $x$ amounts to solving a nested time discretization with time-discrete mean curvature motion as the inner and the actual time-discrete Willmore flow as the outer problem,
It can also be understood as a PDE-constrained optimization problem, where one optimizes the energy \autoref{eq:nested_willmore_step} subject to the constraint that $y_{\tilde \tau}[x]$ is the solution of the linear discrete semi-implicit backward Euler scheme for mean curvature motion. 

%%%%%%%% phase field
In this work, we pick up the approach proposed by Franken \etal \cite{FrRuWi11} and focus on the corresponding scheme for phase fields.
%%%%%%%%%%%%%%%%%%%%%%% added
In \cite{FrRuWi11}, Franken \etal observed increased robustness of their nested scheme compared to a discrete semi-implicit approximation of phase field Willmore flow, based on an energy of the form
\begin{equation}
\label{eq:PfApproximationWillmoreEnergy}
\tfrac{1}{2\eps} \int_\Omega (-\eps \laplace u + \tfrac{1}{2\eps} \Psi'(u))^2 \d x
\end{equation}
as proposed in \cite{DuWa07}.
%%%%%%%%%%%%%%%%%%%%%%% end added
Our core idea is to construct a hybrid scheme where the inner problem is solved using a neural network, 
whereas the outer problem remains a classical minimizing movement optimization scheme. 
For the inner problem, we will build on recent work by Bretin \etal \cite{BrDeMa22} on neural operators for time discrete mean curvature flow in a phase field formulation.
These neural operators are efficient to evaluate, straightforward to differentiate, and convergent under spatial refinement in numerical experiments. 
Using them in the variational time-discretization, we obtain a numerical scheme 
for the phase field approximation of Willmore flow that, similar to the original nested scheme, 
remains stable for large time steps while being computationally more efficient.

By adopting this hybrid strategy, we take a step towards a fully neural treatment of Willmore flow,
where the ultimate goal is to learn the solution operator itself. 
Indeed, considering a network that approximates a minimizing movement scheme for the phase field approximation of the Willmore energy \autoref{eq:PfApproximationWillmoreEnergy} would be a promising alternative.
However, it requires sufficient training data for a suitable large set of initial conditions, which has to be computed in the training phase.
Given the usual computing times for Willmore flow, this is surely a major computational challenge. Furthermore, the selection of a suitable class of initial conditions is still open.
The advantage of our approach is that the nested scheme by Franken \etal \cite{FrRuWi11} allows for
an approximation of the mean curvature based on a neural operator for the approximation of MCF. 
This only requires training data for a corresponding phase field approximation of MCF. 
Such training data is easy and computationally cheap to generate, as discussed by Bretin \etal \cite{BrDeMa22}.
Indeed, they showed that training on hypersphere evolutions is sufficient.

Within this context, the integration of a learned mean curvature flow operator into the nested framework provides a practical and principled intermediate step: it leverages the efficiency of neural operators while retaining the well-established structure of the nested algorithm. 
We show that this integration indeed leads to a viable hybrid scheme.

%%%%%%%% organization 
The remainder of the paper is organized as follows. 
In \autoref{sec:method}, we recapitulate the adaptation of the variational time-discretization to phase fields from \cite{FrRuWi11} and introduce a neural operator approximation of mean curvature flow for phase fields inspired by \cite{BrDeMa22}. 
We combine both to obtain the spatially discrete hybrid scheme for Willmore flow in \autoref{sec:discretization}.
In \autoref{sec:experiments}, we first experimentally validate the convergence properties of the neural network-based discrete mean curvature flow to then underpin a corresponding validation of our hybrid Willmore flow scheme.  
Furthermore, we use our hybrid scheme to compute the evolution for different interesting initial curves in 2D and surfaces in 3D. 
Afterwards, in \autoref{sec:applications}, we apply the scheme for curve and surface restoration. 
Finally, in \autoref{sec:conclusions}, we briefly draw conclusions.

%%%%%%%%%%%%%%%%%%%%%%%%%%%%%%%%
%%%%%%%%%%%%%%%%%%%%%%%%%%%%%%%%
%%%%%%%%%%%%%%%%%%%%%%%%%%%%%%%%
\section{Synthesis of the time-discrete Willmore flow}
\label{sec:method}
Our scheme has two essential ingredients: 
a {\it phase field based minimizing movement scheme for Willmore flow} as introduced by Franken \etal \cite{FrRuWi11}, and a {\it convolution based approximation of mean curvature}. 
Below, we will first discuss the relevant parts of the former, then detail the latter and describe how both are combined for the purpose of a robust and efficient approximation of Willmore flow.

%%%%%%%%%%%%%%%%%%%%%%%%%%%%%%%%

%% phase field approximation 
\paragraph{Phase field based minimizing movement scheme for Willmore flow.} 
Following Franken \etal \cite{FrRuWi11}, we assume that hypersurfaces under consideration are represented by Modica--Mortola-type phase field functions \cite{MoMo77} % in $L^2(\Omega)$ 
with periodic boundary conditions.
In what follows, we assume that all functions and interfaces are sufficiently smooth.
We consider the interfacial energy
\begin{equation} \label{eq:aeps}
\Per^\eps[v] \coloneqq \frac{1}{2} \int_\Omega \eps \lvert \nabla v \rvert^2 + \frac{1}{\eps} \Psi(v) \d x
\end{equation}
with the double well potential \(\Psi(v) = \frac{9}{16}(1 - v^2)^2\). 
Modica and Mortola \cite{MoMo77} showed that the \(\Gamma\)-limit of \(\Per^\eps[\cdot]\) in the \(L^1\) topology is half the total variation of a function \(u\in BV(\Omega;\{-1,1\})\) in \(\Omega\), \ie the perimeter of the set \(\{u=1\}\).

Phase field functions minimizing \eqref{eq:aeps} follow an optimal profile in normal direction across the boundary of a set $\omega\subset \Omega$, which is \(\tilde v_\eps \colon \R \to \R; \, s \mapsto \tanh(- \frac{3s}{4\eps})\). 
Thus, \(\tanh \left(- \frac{3}{4\eps} \sdist(y, \omega) \right)\) is the 
optimal phase field profile of an interface \(\Gamma  = \partial \omega \cap \Omega\) 
for fixed \(\eps\), 
where the signed distance function of $\Gamma$ is defined as 
\(\sdist(x,\omega) = \dist(x,\omega) - \dist(x, \omega^c)\) with 
$\omega^c$ being the complement of $\omega$.
Now, let \(u^k_\eps\) denote the phase field representations of the hypersurface \(\Gamma^k\) with parametrization \(x^k\), which is the boundary of a subset \(\omega^k \subset \Omega\), \ie \(\Gamma^k = \partial \omega^k \cap \Omega\). 
Similarly, let \(u_\eps\) be the representation of  \(\Gamma = \partial \omega \cap \Omega\) with parametrization \(x\), and \(v_{\tilde \tau}[u_\eps]\) a phase field representation of the image of \(\Gamma\) under timestep of mean curvature flow with step size $\tilde \tau$ and initial data $u_\eps$. 
Following Franken \etal \cite{FrRuWi11}, we take into account a further minimizing movement scheme to define the operator \(v_{\tilde \tau}[\cdot]\) variationally and obtain
\begin{equation} \label{eq:implicitMCF}
v_{\tilde \tau}[u] \coloneqq \argmin_{v\in H^1(\Omega)} \left(\eps \|u - v\|^2_{L^2(\Omega)} + 2 \tilde \tau \Per^\eps[v]\right).
\end{equation}

To translate \eqref{eq:nested_willmore_step} to the phase field context, one observes
for a shift \(\delta\) of the optimal profile $\tilde v_\eps$ in one dimension
%\begin{align*}
%\eps  \int_\R &\left(\tilde v_\eps(t)-\tilde v_\eps(t-\delta)\right)^2 \d t \leq \frac{4}{3} \delta^2 (1 + \Theta(\delta, \eps))\,,
%\end{align*}
\begin{align*}
\eps  \int_\R &\left(\tilde v_\eps(s)-\tilde v_\eps(s-\delta)\right)^2 \d s = \delta^2 (1 + \Theta(\delta, \eps))\,,
\end{align*}
where \(\Theta(\delta, \eps) = C \left( \frac{\delta}{\eps} + \frac{\delta^2}{\eps^2} \right)\). 
A detailed calculation is given in \cite{FrRuWi11}.
Chosing \(\delta = \eps^\beta\) with \(\beta>1\) implies \(\Theta(\delta, \eps) \leq C \delta^{(1-\frac1\beta)}\) and thus
\begin{align*}
\eps \int_\Omega \left( u_\eps(x+\delta(x) n(x))- u_\eps(x)\right)^2 \d x &=
\int_\Gamma \delta^2(x) \Hd^{d-1} \left(1+ O(\Theta(\lVert\delta\rVert_{\infty}, \eps))\right)\,,
\end{align*}
where now \(\delta\) is some function on \(\Gamma\), \(n\) is the normal field of \(\Gamma\) and \(\delta\) is assumed to be extended constantly in normal direction to \(\Gamma\). 
Next, assuming that all involved phase field functions \(u^k\), \(u\), and \(v_{\tilde \tau}[u]\) follow the optimal profile, one observes that
%\begin{align*}
%    \eps \lVert v_{\tilde \tau}[u] - u \rVert_{L^2(\Omega)}^2 &= \frac{4}{3} \lVert y_{\tilde \tau}[x] - x\rVert_{L^2(\Gamma)}^2 \left(1 + O(\Theta(\lVert y_{\tilde \tau}[x] - x\rVert_{L^\infty(\Gamma)}, \eps))\right)\,, \\
%    \eps \lVert u - u^k \rVert_{L^2(\Omega)}^2 &= \frac{4}{3}  \lVert x - x^k \rVert_{L^2(\Gamma^k)}^2 \left(1 + O(\Theta(\lVert x-x^k\rVert_{L^\infty(\Gamma^k)}, \eps))\right)
%\end{align*}
\begin{align*}
\eps \lVert v_{\tilde \tau}[u] - u \rVert_{L^2(\Omega)}^2 &=  \lVert y_{\tilde \tau}[x] - x\rVert_{L^2(\Gamma)}^2 \left(1 + O(\Theta(\lVert y_{\tilde \tau}[x] - x\rVert_{L^\infty(\Gamma)}, \eps))\right), \\
\eps \lVert u - u^k \rVert_{L^2(\Omega)}^2 &=  \lVert x - x^k \rVert_{L^2(\Gamma^k)}^2 \left(1 + O(\Theta(\lVert x-x^k\rVert_{L^\infty(\Gamma^k)}, \eps))\right).
\end{align*}
The above estimates were presented in \cite{FrRuWi11} using the double well function $\Psi(v) = (1 - v^2)^2$ and thus consistently with an additional factor $\tfrac43$ (and optimal profile $\tilde v_\eps(s) = \tanh(-s\eps^{-1})$).

Finally, with these approximations at hand, 
one can define the energy 
\begin{equation}
\label{eq:discrete_willmore_flow_energy}
e^\eps[u^k,u] =  \eps\lVert u - u^k \rVert_{L^2(\Omega)}^2+\frac{\tau\eps}{\tilde\tau^2} \lVert v_{\tilde \tau}[u] - u \rVert_{L^2(\Omega)}^2
\end{equation}
for two functions $u^k$ and $u$ considered as phase field descriptions of $\Gamma^k$ and $\Gamma$.
By our above estimates, this energy is equivalent to the energy \(e[\cdot]\) associated with the variational time-discretization in \autoref{eq:nested_willmore_step}.

Altogether for sufficiently small phase field parameter \(\eps\) and sufficiently small time step sizes \(\tau\), \(\tilde \tau\), this leads to the following nested variational time discretization of Willmore flow:
\begin{definition}[Variational time discretization of Willmore flow \cite{FrRuWi11}]\label{def:minProblem}
	For $e^\eps[\cdot,\cdot]$ defined in \autoref{eq:discrete_willmore_flow_energy} based on some mapping 
	$v_{\tilde \tau} \colon L^2(\Omega) \to L^2(\Omega)$ with $\tau, \; \tilde \tau >0$
	and some $u^0 \in L^2(\Omega)$ 
	we iteratively compute 
	\begin{equation}
	\label{eq:discrete_willmore_step}
	u^{k+1} = \argmin_{u \in L^2(\Omega)} e^\eps[u^k, u]
	\end{equation}
	as the time discrete phase field solution at time \((k+1)\tau\) for $k\in\N_0$.
\end{definition}
%In our above approximation arguments \(\delta\) plays the role of local distance between the evolving hypersurfaces at two consecutive time steps either of mean curvature motion or of Willmore flow. 
%Hence, the approximation is effectively valid if these distances are sufficiently smaller than the phase field parameter \(\eps\).

%% transition
\paragraph{Convolution based approximation of mean curvature.}
Thus far, we have defined \(v_{\tilde \tau}[u]\) variationally, and obtained an inner variational problem to be solved for every $u$ in the outer problem \autoref{eq:discrete_willmore_step}.
Instead of solving the Euler-Lagrange equation associated with \autoref{eq:implicitMCF} for every function $u$ arising in the outer problem, we pick up the approach by Bretin \etal \cite{BrDeMa22}.
To this end, we define a spatially continuous neural operator \(v^{f,\kernel}_{\tilde \tau}\colon L^2(\Omega) \to L^2(\Omega)\) that will be trained to approximate \({v}_{\tilde \tau}[\cdot]\). We follow the structure proposed for the discrete context in \cite{BrDeMa22} and combine a nonlinear activation function and a convolution operator to define the neural operator \(v^{f,\kernel}_{\tilde \tau}\).
This leads to the definition of the still spatially continuous neural operator  
\begin{equation}
\label{eq:mbo_structure}
v^{f,\kernel}_{\tilde \tau}[u] \coloneqq f_{\tilde \tau} \circ (\kernel_{\tilde \tau} * u),
\end{equation}
where one first applies a convolution kernel  \(\kernel_{\tilde \tau} \in L^2(\R^d) \), followed by the concatenation with a nonlinear activation function  \(f_{\tilde \tau} \in  C^0(\R)\).
When applying the kernel $\kappa_{\tilde \tau}$ to $u$, we assume that $u$ is periodically extended to all of $\R^d$.
The convolution $\kappa_{\tilde \tau}  \ast u$ with $\kappa_{\tilde \tau} \in L^2(\R^d)$ is in $C^0(\R^d)$.
Thus, $v^{f,\kernel}_{\tilde \tau}[u] \in C^0(\R^d)$ as well and can in particular be evaluated point-wise.

When using the neural operator \(v^{f,\kernel}_{\tilde \tau}\) in the nested scheme \autoref{eq:discrete_willmore_flow_energy}, we obtain a time-discrete but spatially still continuous scheme.
It combines a standard optimization problem over functions in $L^2(\Omega)$ with a neural operator acting on functions in $L^2(\Omega)$.

%%%%%%%%%%%%%%%%%%%%%%%%%%%%%%%%%%%%%%%%%%%%%%%%%%%%%%%%%%%%%%%%%
%motivation: MBO scheme
\begin{remark}[Relation to the Merriman--Bence--Osher scheme and semi-implicit time stepping for Allen-Chan flow]
	To motivate the neural network architecture used by Bretin \etal we first observe that 
	the rescaled $L^2$-gradient flow of $\Per^\eps$ with time discretization \autoref{eq:implicitMCF} 
	is given by the Allen--Cahn equation
	\(  \partial_t u  -\laplace u + \tfrac1{2\eps^2} \dwell^\prime(u)=0.\) 
	Evaluating the nonlinearity \(\dwell^\prime\) implicitly at \(u^{k+1}\) 
	and the Laplace operator explicitly at \(u^k\) we obtain the time-discrete equation
	\begin{equation}\label{eq:discreteAllenCahn}
	\tfrac{u^{k+1} - u^k}{\tilde \tau} - \laplace u^k + \tfrac{1}{2\eps^2} \dwell^\prime(u^{k+1}) = 0
	\end{equation}
	to iteratively compute the sequence \((u^k)_{k\in\N}\) of phase fields given an initial phase field \(u^0\). 
	The function \(\phi_{\eps, \tilde \tau}(u) = u + \frac{\tilde{\tau}}{2\eps^2} \dwell^\prime(u)\) is monotone for \(\frac{\tilde{\tau}}{\eps^2} < \frac{8}{9}\) and thus invertible.
	Hence, \autoref{eq:discreteAllenCahn} can be rewritten as 
	\(u^{k+1} = \phi_{\eps, \tilde \tau}^{-1}((\Id +  \tilde \tau \laplace) u^k)\) and finally using  
	the approximation \((\Id +  \tilde \tau \laplace) u^k = e^{-\tilde\tau \laplace} u^k + O(\tilde\tau^2)\) 
	we obtain
	\(u^{k+1} = \phi^{-1}(e^{-\tilde\tau \laplace} u^k).\)
	%begin alternative:
	%we obtain the
	%\begin{equation} \label{eq:ac_splitting}
	%u^{k+1} = \phi^{-1}((\Id - \eps \tilde \tau \laplace)^{-1} u^k).
	%\end{equation}
	%The implicit time step for the heat flow \((\Id - \tilde \tau \laplace)^{-1} u^k\) can 
	%be approximated via convolution with a Gaussian kernel leading to a concatenation of a nonlinear function and a convolution kernel as the approximation scheme of a small time step of mean curvature flow.
	%end alternative
	This indeed reflects the structure \autoref{eq:mbo_structure} proposed by Bretin \etal \cite{BrDeMa22} 
	and resembles the Merriman--Bence--Osher (MBO) scheme \cite{MeBeOs92,MeBeOs94} for characteristic functions, where one first solves the linear heat equation with time step \(\tilde \tau\) for a characteristic function as the initial data and then applies a thresholding function to obtain the time step updated characteristic function. 	In this sense, the application of the function \(\phi_{\eps, \tilde \tau}^{-1}\) acts as a soft thresholding.
	For the convergence to mean curvature motion, we refer to \cite{Ev93} and \cite{BaGe95}.
	Recently, Budd and van Gennip \cite{BuGe20} studied a semi-implicit time-discretization for the double obstacle	Allen--Cahn equation on graphs and proved that the MBO scheme coincides with a particular choice of a semi-implicit scheme for Allen--Cahn flow. 
	The scheme takes the form  \(u^{k+1} = \rho \circ \left(e^{-\tilde\tau \Delta} u^{k}\right)\), with $\Delta$ being the (positive definite) graph Laplacian, \(\rho\) a monotone Lipschitz continuous (activation) function, 
	and $e^{-\tilde\tau \Delta}$ the linear operator representing one timestep of the heat flow with time step size \(\tilde\tau\).
	The authors explicitly refer to the analogy of a convolutional neural network in \cite[footnote 6]{BuGe20}.
	%We remark the curiosity that, thus expressed, this scheme can be seen to 
	% resemble a forward pass through a neural net, with diffusion playing the role of the linear operator and ?? the activation function.
\end{remark}
%%%%%%%%%%%%%%%%%%%%%%%%%%%%%%%%%%%%%%%%%%%%%%%%%%%%%%%%%%%%%%%%%
% optimization problem
As proposed in \cite{BrDeMa22}, both the kernel \(\kernel_{\tilde \tau}\) and the scalar function \(f_{\tilde \tau} \in  C^0(\R)\) are learned from data. 
Bretin \etal suggested using (discrete) phase field profiles for the explicitly known evolution of hyperspheres in \(\R^d\) as training data. 
Indeed, the radius of a hypersphere evolving under mean curvature flow with initial radius $r_0$ at time \(\tilde \tau\) is given by 
\begin{equation}
\label{eq:analyticSolutionMCFCircle}
R(r_0, \tilde \tau)=\sqrt{r_0^2-2(d-1)\tilde \tau}\,.
\end{equation}
For a hypersphere of radius $r$, we define the corresponding phase field profile by
\begin{equation}
\label{eq:analyticSolutionMCFCirclePf}
u_r(x) \coloneqq \tanh \left(\frac{3(r- | x |)}{4\eps}\right)\,.
\end{equation}
The aim is to learn  \(f_{\tilde \tau}\) and \(k_{\tilde \tau}\) such that, for a range of radii \(r\), the phase field representation of a hypersphere of radius $r$ is mapped to one of a hypersphere of radius $R(r, \tilde \tau)$, \ie~\(v^{f,\kernel}_{\tilde \tau}[u_r] \approx u_{R(r, \tilde \tau)}\).
This way, $v^{f,k}_{\tilde \tau}$ approximates the time-discrete solution operator of mean curvature flow on the chosen training data. 
We formulate  an optimization problem by turning the desired relation into a least-squares loss, and thus determine $f_{\tilde \tau}$ and $k_{\tilde \tau}$ as solutions of

\begin{equation}
\label{eq:mcf_scheme_optimization}
\min_{f,\kernel} \int_{r_\mathrm{min}}^{r_\mathrm{max}} \rVert v^{f,\kernel}_{\tilde \tau}[u_r] - u_{R(r, \tilde \tau)}  \lVert^2_{L^2(\R^d)} \d r,
\end{equation}
where one optimizes over suitable classes of nonlinear functions $f$ and convolution kernels $\kappa$ for a suitable choice of 
\(r_\mathrm{min}\) and \(r_\mathrm{max}\) .

\begin{remark}
	In their study, Bretin \etal \cite{BrDeMa22} focused on a discrete formulation on regular grids and directly developed a time and space discrete evolution operator for mean curvature motion. 
	The underlying neural network consists of a single convolution layer and a scalar activation function realized with a multilayer perceptron.
	Bretin \etal demonstrated empirically that this also leads to a reasonable approximation of mean curvature flow for fairly general initial data and corresponding phase fields.  
	Functions of the form \autoref{eq:mbo_structure} are efficient to evaluate and differentiate when implemented on discrete grids using neural networks.
	We will retrieve this discrete formulation in the next section.
\end{remark}

Furthermore, employing the scheme composed of \autoref{eq:discrete_willmore_flow_energy} and \autoref{eq:mbo_structure}, we can show the existence of time-discrete Willmore flow, \ie the existence of minimizers of \autoref{eq:discrete_willmore_flow_energy} for fixed scalar function\(f\) and convolution kernel $\kernel$:
\begin{proposition}
	\label{prop:existence}
	For $\Omega=(0,1)^d$, $u^0\in L^2(\Omega)$ periodically extended on $\R^d$, and \(e^\eps[u^0,\cdot]\) as in \eqref{eq:discrete_willmore_flow_energy} with \(v_{\tilde \tau}\) of the form \(v_{\tilde \tau}[u] = f(\kernel\ast u)\) for fixed \(f\in C(\R)\) and $\kernel \in L^2(\R^d)$ there exists a minimizer \(u\in L^2(\Omega)\) of \(e^\eps[u^0,\cdot]\) in the class of periodically extended $L^2(\Omega)$ functions.
\end{proposition}
\begin{proof}
	Let \((u_j)_j\subset L^2(\Omega)\) be a minimizing sequence of $e^\eps[u^0,\cdot]$, with
	\begin{align*}
	C = e^\eps[u^0, u^0] \geq e^\eps[u^0, u_j] \geq \eps \|u^0 - u_j\|^2_{L^2(\Omega)} 
	\geq \frac{\eps}{2} \|u_j\|^2_{L^2(\Omega)} - \eps \|u^0\|^2_{L^2(\Omega)}\,,
	\end{align*}
	for a constant $C$ depending on $\Vert u^0\Vert_{L^2(\Omega)}$. Here, we used that 
	$\Vert \kernel\ast u_0 \Vert_{C^0(\bar \Omega)}$ is bounded and applied Young's inequality.
	Hence, \(\Vert u_j\Vert_{L^2(\Omega)}\leq \hat C\) for some constant $\hat C$, 
	and thus there exists \(u\in L^2(\Omega)\), 
	and a subsequence (not relabeled), with \(u_j\rightharpoonup u\) weakly in \(L^2(\Omega)\). 
	Once more using that \(\kernel \in L^2(\R^d)\), we have 
	\begin{align*}
	(\kernel \ast u_j)(x) = \int_\Omega \kernel(x - y) u_j(y) \d y \rightarrow \int_\Omega \kernel(x - y) u(y) \d y = (\kernel\ast u)(x) 
	\end{align*}
	for $j\rightarrow \infty$.
	Furthermore, we observe
	\begin{align*}
	\sup_{x\in\Omega}\left| (\kernel \ast u_j)(x) \right| \leq \|\kernel\|_{L^2(\R^d)}\|u_j\|_{L^2(\Omega)}  \leq \hat C  \|\kernel\|_{L^2(\R^d)}.
	\end{align*}
	Since \(f\) is uniformly continuous on \([-\hat C \|\kernel\|_{L^2(\R^d)},\hat C\|\kernel\|_{L^2(\R^d)}]\), we obtain  \(f(\kernel\ast u_j) \rightarrow f(\kernel\ast u)\) pointwise and in $L^2(\Omega)$. Finally, lower semicontinuity of the $L^2$ norm implies
	\begin{align*}
	\inf_{\tilde u\in L^2(\Omega)}e^\eps[u^0, \tilde u] = \liminf_{j\rightarrow \infty} e^\eps[u^0, u_j] \geq e^\eps[u^0, u].
	\end{align*}
	Thus \(u\in L^2(\Omega)\) and minimizes \(e^\eps[u^0,\cdot]\) in the class of periodically extended $L^2(\Omega)$ functions.
\end{proof}

%%%%%%%%%%%%%%%%%%%%%%%%%%%%%%%%
%%%%%%%%%%%%%%%%%%%%%%%%%%%%%%%%
%%%%%%%%%%%%%%%%%%%%%%%%%%%%%%%%
\section{Spatial discretization}
\label{sec:discretization}

% discretization in space
\paragraph{Minimizing movement scheme for Willmore flow.}
To discretize the time-discrete Willmore flow in $\R^d$ based on the minimizing movement scheme \autoref{eq:discrete_willmore_flow_energy}, we consider a regular grid on \([0,1]^d\) with gridsize $h=\tfrac1n$ for $n\in\N$ and nodes $x_\alpha = (\tfrac{\alpha_1}{n},\cdot, \tfrac{\alpha_d}{n})$ for a multi-index $\alpha$ in the multi index set $\mathcal{I}_n \coloneqq (\{0,\ldots, n\})^d$.
On this grid we consider functions represented  by nodal vectors $U\in \R^{\vert \mathcal{I}_n\vert}$ with function values $U_{\alpha}$ at nodes $x_\alpha$.
A discrete \(L^2\) norm is defined as
\begin{align*}
\lVert U \rVert_{L^2} \coloneqq  \sqrt{(n+1)^{-d} \sum_{\alpha\in \mathcal{I}_n} (U_\alpha)^2},
\end{align*}
\ie the square root of the average of the squared entries of \(U\).
Thus, the discrete counterpart of the energy \autoref{eq:discrete_willmore_flow_energy} is given by
\begin{equation}
\label{eq:fully_discrete_willmore_flow_energy}
E^\eps[U^k,U] = \eps\lVert U - U^k \rVert_{L^2}^2+\frac{\tau\eps}{\tilde\tau^2} \lVert V_{\tilde \tau}[U] - U \rVert_{L^2}^2,
\end{equation}
where \(V_{\tilde \tau}[U]\) denotes a spatially discrete counterpart of \autoref{eq:mbo_structure}.

\paragraph{Neural approximation of discrete mean curvature flow.}
Here, we follow the approach by Bretin \etal \cite{BrDeMa22} and consider discrete kernels and nonlinear activation functions, which are defined themselves as networks.
For a discrete kernel $\NNkernel \in \R^{\Z^d}$, one defines the nodal vector resulting from the discrete convolution with this kernel as 
\begin{align*}
(\NNkernel\ast U)_\alpha \coloneqq \sum_{\beta \in \Z^d} \NNkernel_\beta U_{\alpha+\beta}.
\end{align*}
Here, we assume periodicity of the discrete function to which we apply the discrete convolution, \ie \(U_{\alpha+\beta n}=U_{\alpha}\) for \(\beta \in \Z^d\).
Now, one considers discrete kernels with a fixed width $n_\NNkernel$ and a \(n_\NNkernel^d\) stencil, \ie \(\NNkernel_{\alpha} =0\) for $\max_{i=1,\ldots,d} \vert \alpha_i \vert > \tfrac{n_\NNkernel-1}{2}$. 
To ensure consistency with the continuous mean curvature motion, the necessary size \(n_\NNkernel\) of the kernel depends on the time step size $\tilde \tau$ and the grid size $h$.
At the same time, smaller kernels are more efficient to train, which creates a trade-off between accuracy and speed. 
This trade-off will be explored in our numerical experiments below.
The point-wise function \(f\)  is discretized using a fully-connected neural network \(F^\theta \colon \R \to \R\) with $L$ layers and layer sizes \(N_1,\ldots,N_L\). 
The $l$th layer is described in terms of a weight matrix \(W^{l} \in \R^{N_{l} \times N_{l-1}}\) with $N_0=1$ and a bias vector \(b^{l} \in \R^{N_{l}}\).
These degrees are gathered in a parameter vector \(\theta = (W^1,\ldots, W^{L},b^1,\ldots, b^{L})\).
Then, one defines \(F^\theta(s)=s^L\) with \(s^{l} \coloneqq  \rho\left( W^{l} s^{l-1} + b^{l} \right)\) and \(s^0=s\) and the choice \(\rho(s) = \exp(-s^2)\) as the nonlinear activation function. 
In practice, we used six layers with sizes \(32, \,16, \,8, \,4, \,2,\,1\).
For given parameters \(\theta\) and \(\NNkernel\), we obtain the discrete operator 
\begin{equation}
\label{eq:mcf_operator_discrete}
V^{\theta,\NNkernel}_{\tilde \tau}[U] \coloneqq F^\theta(\NNkernel \ast U).
\end{equation}

% neural network training
Now, one approximates the optimization problem \autoref{eq:mcf_scheme_optimization} via our discretization and a sampling of training data.
To this end, one considers $m$ radii \(r_1,\ldots,r_{m}\) sampled uniformly from an interval \([r_\mathrm{min}, r_\mathrm{max}]\) and minimizes the loss functional
\begin{align}
\label{eq:mcf_scheme_optimization_discrete}
\mathcal{L}[\theta,\NNkernel] =  
\frac{1}{m} \sum_{i=1}^{m} \rVert V^{\theta,\NNkernel}_{\tilde \tau}[U_{r_i}] - U_{R(r_i, \tilde \tau)}  \lVert^2_{L^2},
\end{align}
over the total set of parameters \((\theta,\NNkernel)\), where \(U_r\) are nodal evaluations of the hypersphere phase fields \(u_r\).
We approximately solve problem \autoref{eq:mcf_scheme_optimization_discrete} using the Adam optimizer \cite{KiBa15}, with $m=100$, $r_\mathrm{min} = 0.05$, $r_\mathrm{max} = 0.4$, and usually employ mini-batching, \ie approximating the sum in \autoref{eq:mcf_scheme_optimization_discrete} using only ten randomly drawn radii, to speed up the minimization as was proposed by Bretin \etal \cite{BrDeMa22}.
The overall method is described in \autoref{alg:mcf_training}.
The resulting neural phase field operator \(V^{\theta,\NNkernel}_{\tilde \tau}[\cdot]\) approximates the PDE solution \(v_{\tilde \tau}\) defined in \autoref{eq:implicitMCF} and is inexpensive to evaluate.

\begin{algorithm}[h!]
	\caption{Learning the fully discrete mean curvature operator (offline phase) \cite{BrDeMa22}}
	\label{alg:mcf_training}
	\begin{algorithmic}[1]
		\Require Time step size $\tilde \tau$ of mean curvature flow, interface parameter $\eps$, kernel size \(n_\NNkernel\), gridsize \(h=\frac{1}{n}\), number of samples $m$, minimal radius $r_\mathrm{min}$, maximal radius $r_\mathrm{max}$
		\Ensure Discrete convolution kernel $\NNkernel$ and parameter vector \(\theta\)
		\State Create the training data by uniformly sampling $m$ radii $r_1,\ldots,r_m$ from the interval \([r_\mathrm{min}, r_\mathrm{max}]\) and computing the corresponding pairs of discrete phase fields \((U_{r_i}, U_{R(r_i, \tilde \tau)})\)
		\State Initialize kernel $\NNkernel$ and parameter vector \(\theta\) as zero or using interpolation of coarser resolution
		\While{not converged}
		\State \begin{varwidth}[t]{\linewidth} Randomly sample $\tfrac{m}{B}$ mini-batches \(\{i_1,\ldots,i_B\}\) of size \(B\) from \(\{1,\ldots,m\}\) \\ by drawing without replacement\end{varwidth} 
		\ForAll{mini-batches \(\{i_1,\ldots,i_B\}\)}
		\State Compute the loss for the mini-batch
		\begin{equation*}
		\mathcal{L}[\theta,\NNkernel] = \frac{1}{B} \sum_{k=1}^{B} \rVert V^{\theta,\NNkernel}_{\tilde \tau}[U_{r_{i_B}}] - U_{R(r_{i_B}, \tilde \tau)}  \lVert^2_{L^2}
		\end{equation*}
		\State Compute the gradient of the previous loss using back-propagation
		\State Update the parameters \(\theta\) and \(\NNkernel\) using the Adam optimizer
		\EndFor
		\EndWhile
	\end{algorithmic}
\end{algorithm}

%ffffffffffffffffffffffffffffffffffffffffffffffffffffffff
\begin{figure}[t]
	\includegraphics[width=\textwidth]{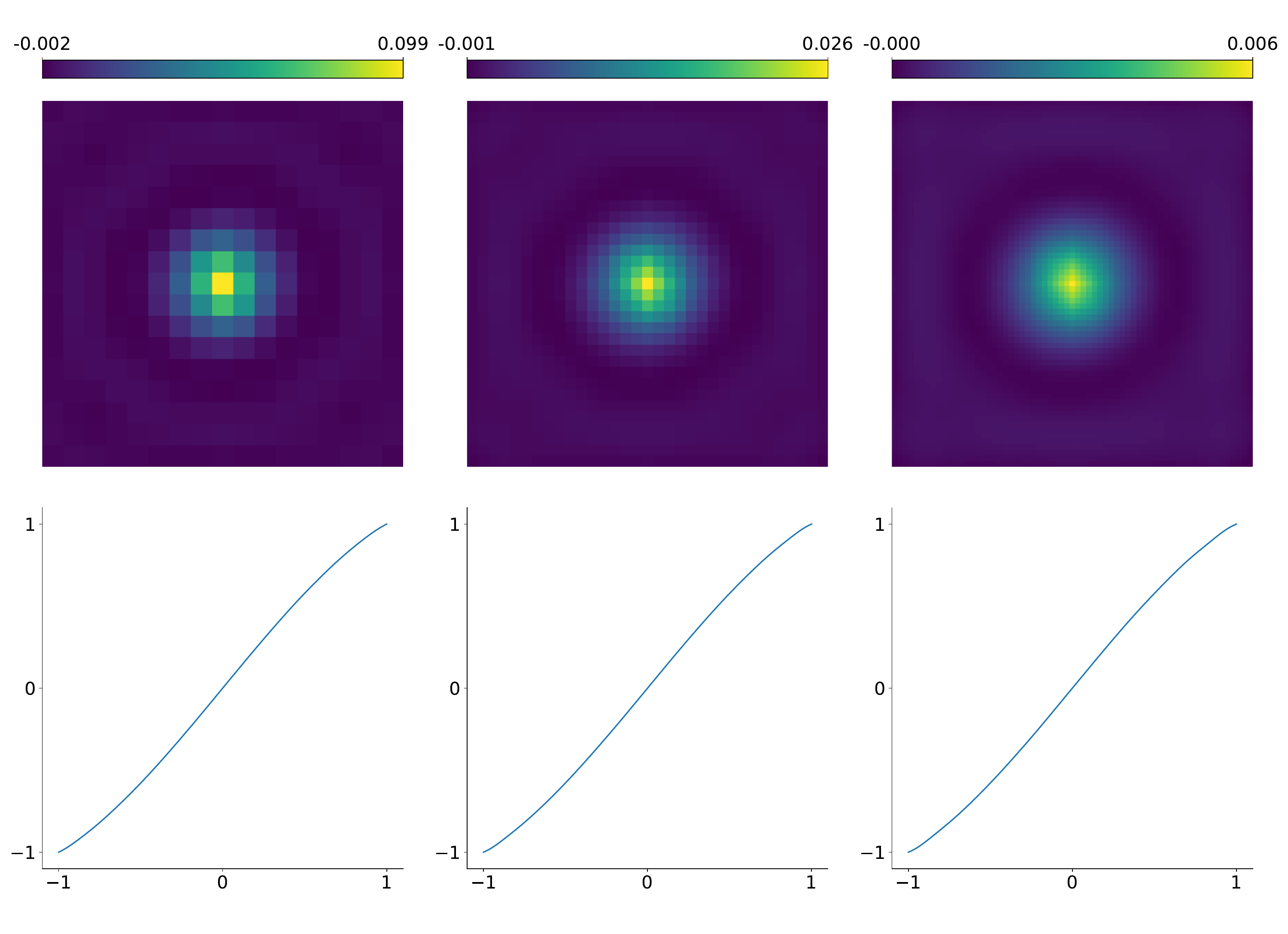}
	\caption{
		Learned networks for \(\eps = 2^{-6}\) and \(\tilde \tau = 2^{-14}\) are displayed via a color coding of the learned kernels \(\NNkernel\) and the graphs of learned activation function \(F^\theta\) on the interval \([-1, 1]\) for increasing $n = 128,\, 256,\, 512$, with stencil widths  $n_K=17,\,33,\,65$, respectively.
	}
	\label{fig:LearnedNetwork}
\end{figure}
%ffffffffffffffffffffffffffffffffffffffffffffffffffffffff

%%%%%
\paragraph{Hybrid scheme for Willmore flow}
We employ Newton's method with Armijo line search to minimize \(E^\eps[U^k,U]\) over \(U\).
Thus, to determine the descent direction \(P\), we approximately solve the linear system \(D^2E^\eps[U^k, \cdot] P = DE^\eps[U^k,\cdot]\) using the conjugated gradient method (see \cite[Chapter 7]{NoWr06}). 
This also prevents us from having to assemble the Hessian. Instead, we only compute the corresponding matrix-vector product. 
Furthermore, we manually implemented the derivatives of \(E^\eps[U^k, \cdot]\) to improve performance.
The overall method is described in \autoref{alg:hybrid_scheme}.
We developed our hybrid method in Python using the PyTorch library \cite{PaGrMa19} and the nonlinear optimization algorithms used in the pytorch-minimize \cite{Fe21} package, which is based on the optimization module of SciPy \cite{ViGoOl20}.

\begin{algorithm}[h!]
	\caption{Hybrid scheme for Willmore flow (online phase)}
	\label{alg:hybrid_scheme}
	\begin{algorithmic}[1]
		\Require Discrete phase field $U^k$ at time $k\tau$ of Willmore flow, discrete convolution kernel $\NNkernel$, parameter vector \(\theta\), time step size $\tau$ of Willmore flow, time step size $\tilde \tau$ of inner mean curvature flow, interface parameter $\eps$
		\Ensure Discrete phase field $U^{k+1}$ at time $(k+1)\tau$ of Willmore flow
		\State Initialize $U^{k+1} = U^k$ 
		\While{not converged}
		\State Evaluate \(V^{\theta,\NNkernel}_{\tilde \tau}[U^{k+1}] = F^\theta(\NNkernel \ast U^{k+1})\) \Comment{\autoref{eq:mcf_operator_discrete}}
		\State Compute \(E^\eps[U^k,U^{k+1}] = \eps\lVert U^{k+1} - U^k \rVert_{L^2}^2+\frac{\tau\eps}{\tilde\tau^2} \lVert V_{\tilde \tau}[U^{k+1}] - U^{k+1} \rVert_{L^2}^2\) \Comment{\autoref{eq:fully_discrete_willmore_flow_energy}}
		\State Compute \(D_2E^\eps[U^k,U^{k+1}]\) and \(D_2^2E^\eps[U^k,U^{k+1}]\) using chain rule
		\State Solve \(D_2^2E^\eps[U^k, U^{k+1}] P = D_2E^\eps[U^k,U^{k+1}]\) using the CG method
		\State Determine step size $t$ via Armijo line search
		\State Set $U^{k´+1} = U^{k+1} + t P$
		\EndWhile
	\end{algorithmic}
\end{algorithm}
Here, $D_2$ and $D^2_2$ indicate the Jacobian and the Hessian with respect to the second argument, respectively.

%%%%%%%%%%%%%%%%%%%%%%%%%%%%%%%%
%%%%%%%%%%%%%%%%%%%%%%%%%%%%%%%%
%%%%%%%%%%%%%%%%%%%%%%%%%%%%%%%%
\section{Numerical experiments}
\label{sec:experiments}
As usual in nonlinear optimization, the chosen initialization when solving \autoref{eq:mcf_scheme_optimization_discrete} impacts the result.
For example, Bretin \etal~\cite{BrDeMa22} successfully trained their network on a resolution of \(n=256\) with stencil width \(n_K=17\), \(\eps=2h\), and \(\tilde \tau=\eps^2\).
They initialized the kernel as zero and the parameters \(\theta\) randomly sampled from a normal distribution.
However, when refining $h$ while keeping \(\eps\) fixed, and thus consistently increasing \(n_K\), we observed a degradation of the approximation quality of \(V^{\theta,\NNkernel}_{\tilde \tau}[\cdot]\).
We mitigated this by first training the network on a coarse resolution as proposed by Bretin \etal~and then progressively passing to finer resolutions. 
In each step, we initialize the kernel using a bilinear interpolation of the coarser one. 
For the nonlinearity \(F^\theta\), we kept the previous parameters as initialization. 
The networks resulting from this training process are displayed in \autoref{fig:LearnedNetwork}.
Using color coding, we show the discrete convolution kernels $K$. We observe an improved radial symmetry under the spatial refinement 
and an overall similarity to heat kernels, as it was also observed by Bretin \etal. This is consistent with the expected smoothing behavior of mean curvature 	flow. The learned scalar function \(F^\theta\) is plotted on the interval \([-1, 1]\), and it is rather stable across different refinement levels.

%ffffffffffffffffffffffffffffffffffffffffffffffffffffffff
\begin{figure}[t]
	\centering
	\includegraphics[width=.9\textwidth]{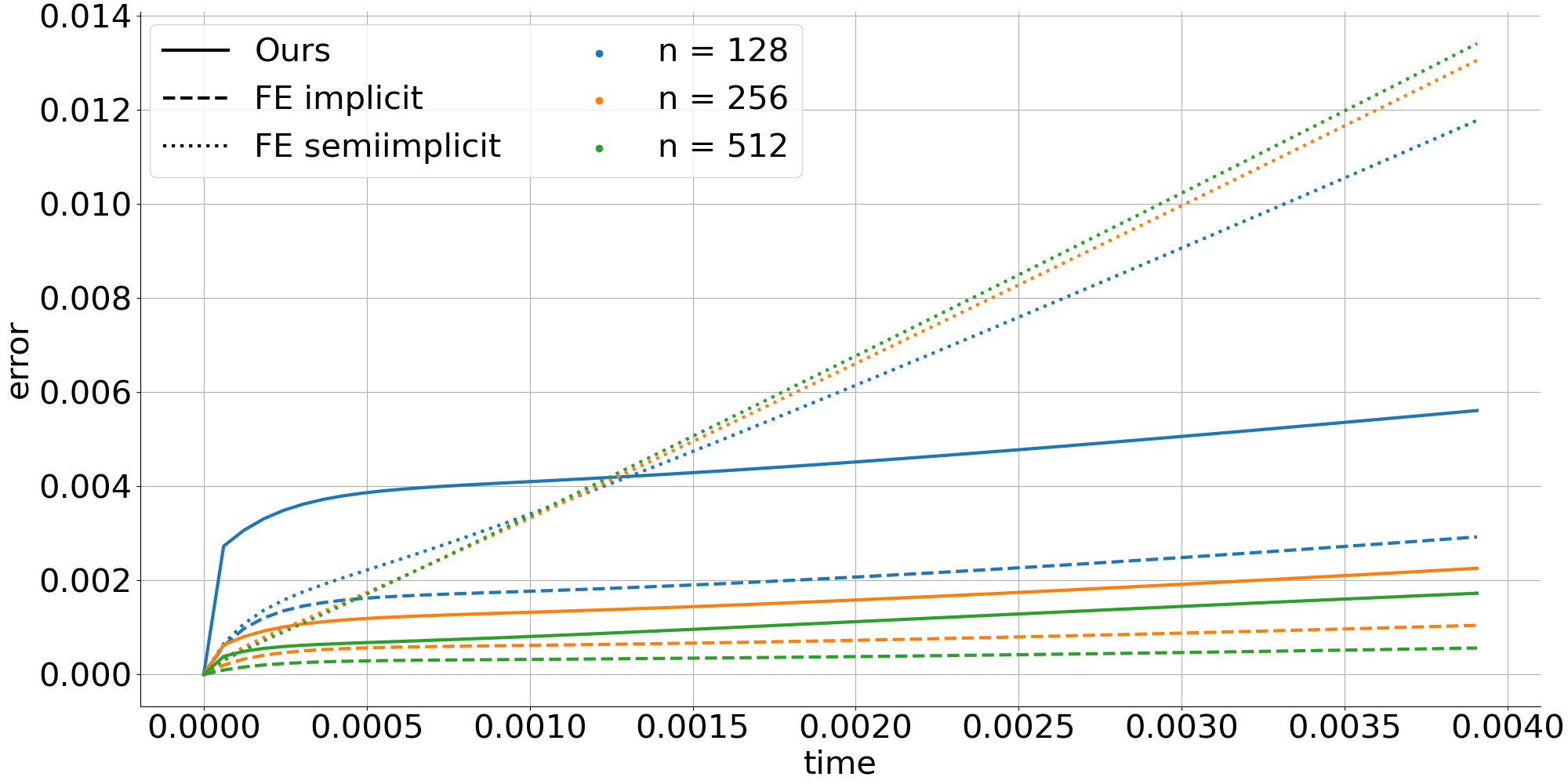}
	\caption{
		Convergence plot for mean curvature flow with fixed \(\eps = 2^{-6}\) and \(\tilde \tau = 2^{-14}\) while increasing \(n\), with \(n_K = \tfrac{n}{8} + 1\). 
		We plot the average $L^2$-error to the analytic solution along time for 30 circles with radii $r_i = 0.05\pi + \tfrac{0.15\pi i}{30}\,$, $i=0,\dots,29$. 
		The line styles correspond to the different methods and the colors to the varying resolution.
		For comparison: for the averaged $L^2$ distance between the solution at time $0.004$ and at the initial time zero, one obtains $\tfrac{1}{30}\sum_{i=1}^{30} \|U_{R(r_i, 0.004)} - U_{R(r_i, 0)}\|_{L^2} \approx 0.133$.
	}
	\label{fig:ConvergenceTestMCF}
\end{figure}

\begin{figure}[t]
	\centering
	\begin{tikzpicture}
	\node[inner sep=0pt] at (0.0, 0.0) {\includegraphics[width=.68\textwidth]{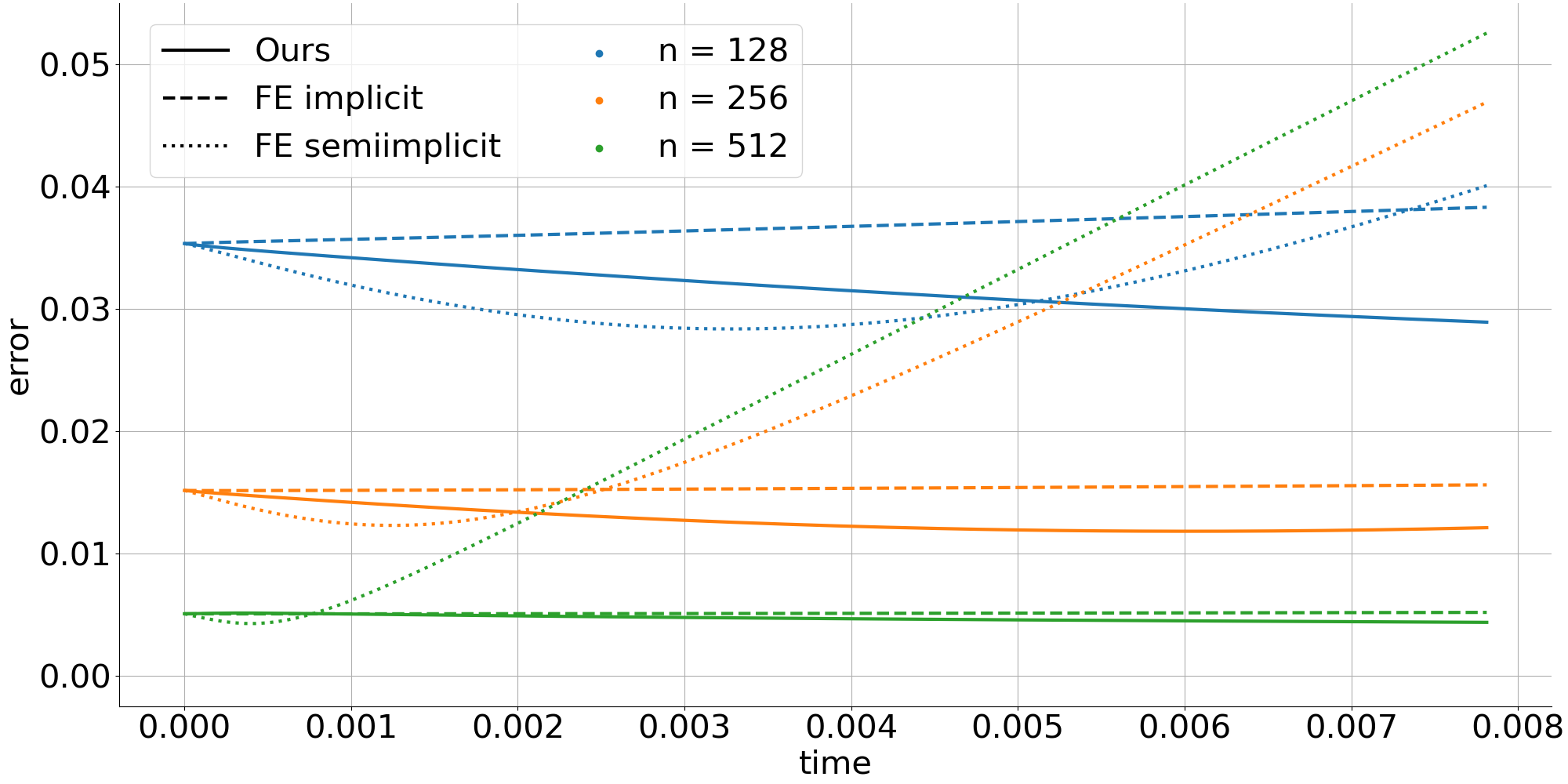}};
	\draw[-] (5.4,-2.47) -- (5.4,2.57);
	\node[inner sep=0pt] at (8, 0.1) {\includegraphics[width=.44\textwidth, trim={0 0cm 0cm 0cm},clip]{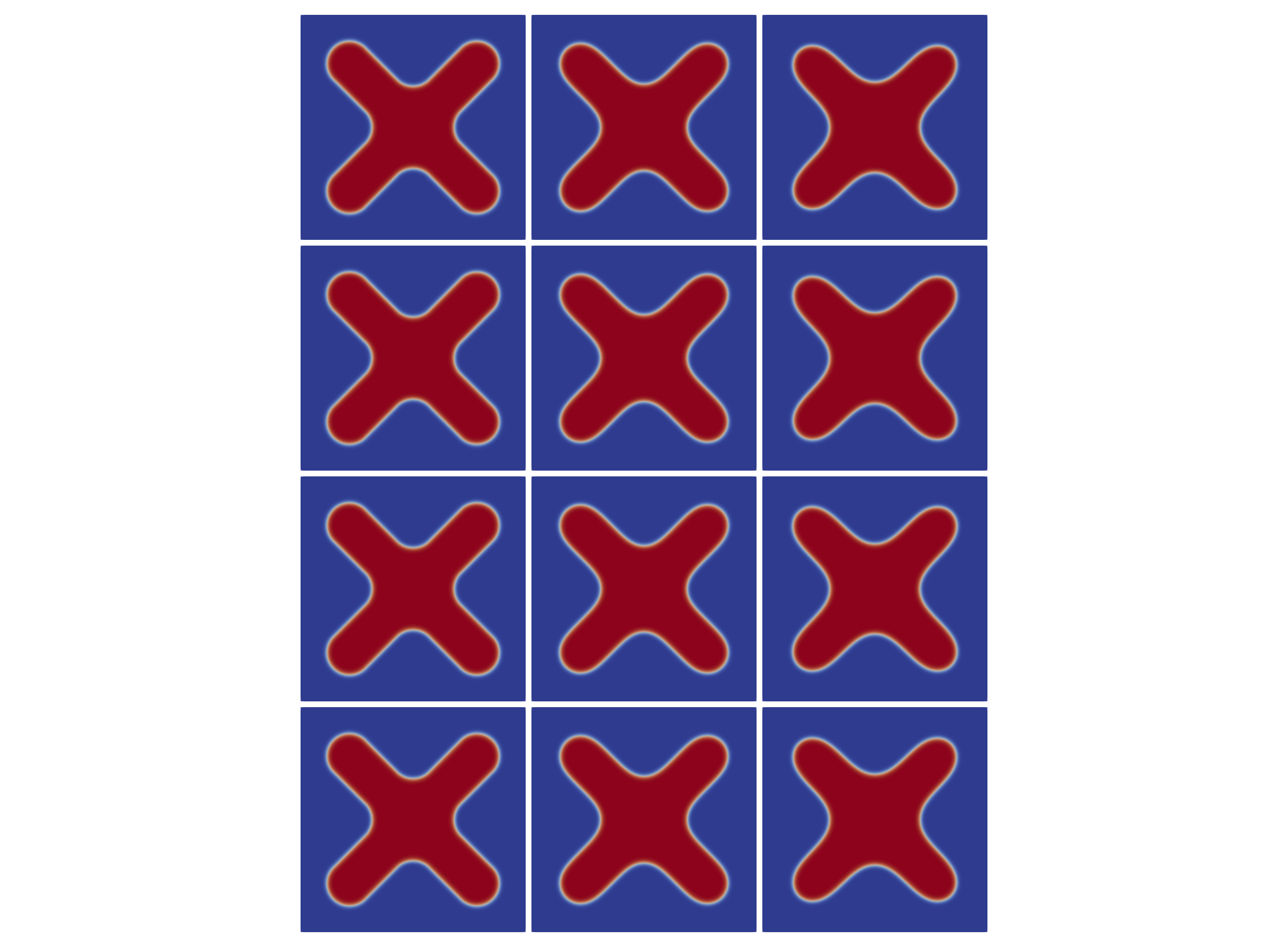}};
	\node[font=\tiny\linespread{0.9}\selectfont, align=center, rotate=90] at (5.8, 2) {Ours\\ $n=256$};
	\node[font=\tiny\linespread{0.9}\selectfont, align=center, rotate=90] at (5.8, 0.8) {FE\\ implicit\\ $n=256$};
	\node[font=\tiny\linespread{0.9}\selectfont, align=center, rotate=90] at (5.8, -0.6) {FE\\ semiimplicit\\ $n=256$};
	\node[font=\tiny\linespread{0.9}\selectfont, align=center, rotate=90] at (5.8, -1.8) {FE\\ implicit\\ $n=2048$};
	\node[font=\tiny\linespread{0.9}\selectfont, align=center] at (6.8, -2.55) {$t=0$};
	\node[font=\tiny\linespread{0.9}\selectfont, align=center] at (8., -2.53) {$t=2^{-8}$};
	\node[font=\tiny\linespread{0.9}\selectfont, align=center] at (9.3, -2.53) {$t=2^{-7}$};
	\end{tikzpicture}
	\caption{
		Left: Convergence plot for mean curvature flow with fixed \(\eps = 2^{-6}\) and \(\tilde \tau = 2^{-14}\) while increasing \(n\), with \(n_K = \tfrac{n}{8} + 1\) on a cross shape. 
		We plot the $L^2$-error to an implicit finite element solution with $n=2048$ and $\tilde \tau = 2^{-16}$. 
		The line styles correspond to the different methods and the colors to the varying resolution.
		Right: computed mean curvature flow evolution of the cross shape for the different schemes at time steps 0 ($t=0$), 64 ($t=2^{-8}$), and 128 ($t=2^{-7}$).
	}
	\label{fig:ConvergenceTestMCFcross}
	
\end{figure}

%\MR{$n \leftrightarrow n_K$?}
%ffffffffffffffffffffffffffffffffffffffffffffffffffffffff
With this setup, we first investigate the convergence of the neural network-based MCF scheme \autoref{eq:mcf_operator_discrete} and \autoref{eq:mcf_scheme_optimization_discrete} in two dimensions under spatial refinement, \ie~we keep the scale parameter \(\eps\) and the time step size \(\tilde\tau\) fixed and increase the spatial resolution and the width of the kernel \(n_K\). 
We compare the neural network scheme by Bretin \etal~\cite{BrDeMa22} with a fully implicit (\cf~\autoref{eq:implicitMCF}) and semi-implicit finite element scheme.
In the latter, one computes the solution \(v\) at the next timestep as the solution of $\frac{v - u}{\tilde \tau} = \laplace v - \tfrac{1}{2\eps^2}  \dwell'(u)$ for given $u$ at the current timestep.
Here, we consider a multi-linear finite element approach on the regular quad mesh.
We perform the validation of convergence for the evolution of circles on the computational domain \(\Omega = (-1,1)^2\).
In \autoref{fig:ConvergenceTestMCF}, we plot the average $L^2$-error over 30 radii ranging from $0.05\pi$ to $0.2\pi$ when comparing to the exact solutions \eqref{eq:analyticSolutionMCFCirclePf} corresponding to circles with radii given by \eqref{eq:analyticSolutionMCFCircle}. 
The error is displayed along 64 timesteps of size $\tilde \tau = 2^{-14}$ for the different schemes and varying resolution.

%ffffffffffffffffffffffffffffffffffffffffffffffffffffffff

\begin{figure}[htbp!]
	
	\begin{tikzpicture}
	\node[inner sep=0pt] at (4.04, 2.5) {\includegraphics[width=.5\textwidth]{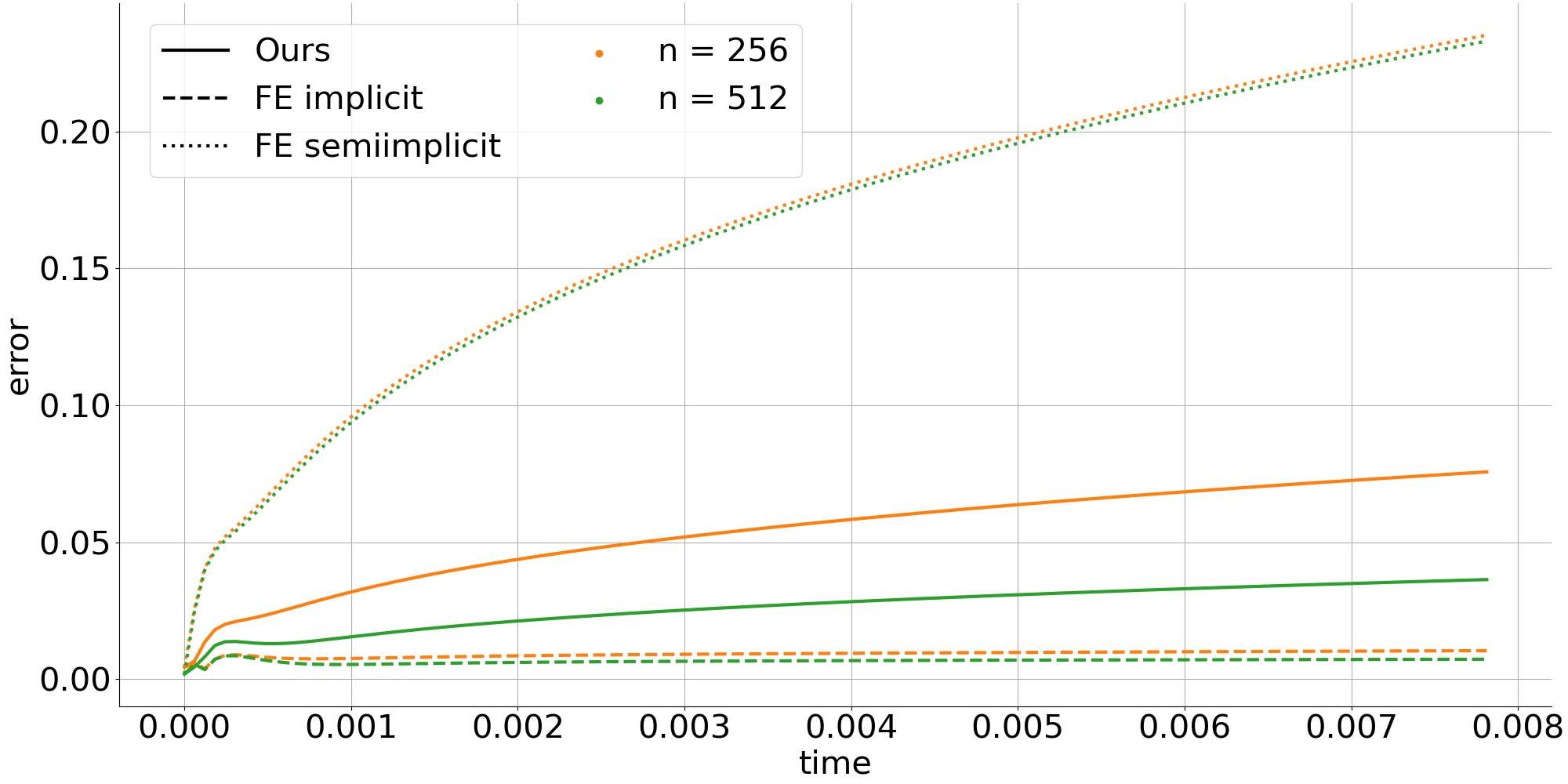}};
	\node[inner sep=0pt] at (-3.74, 2.5) {\includegraphics[width=.5\textwidth]{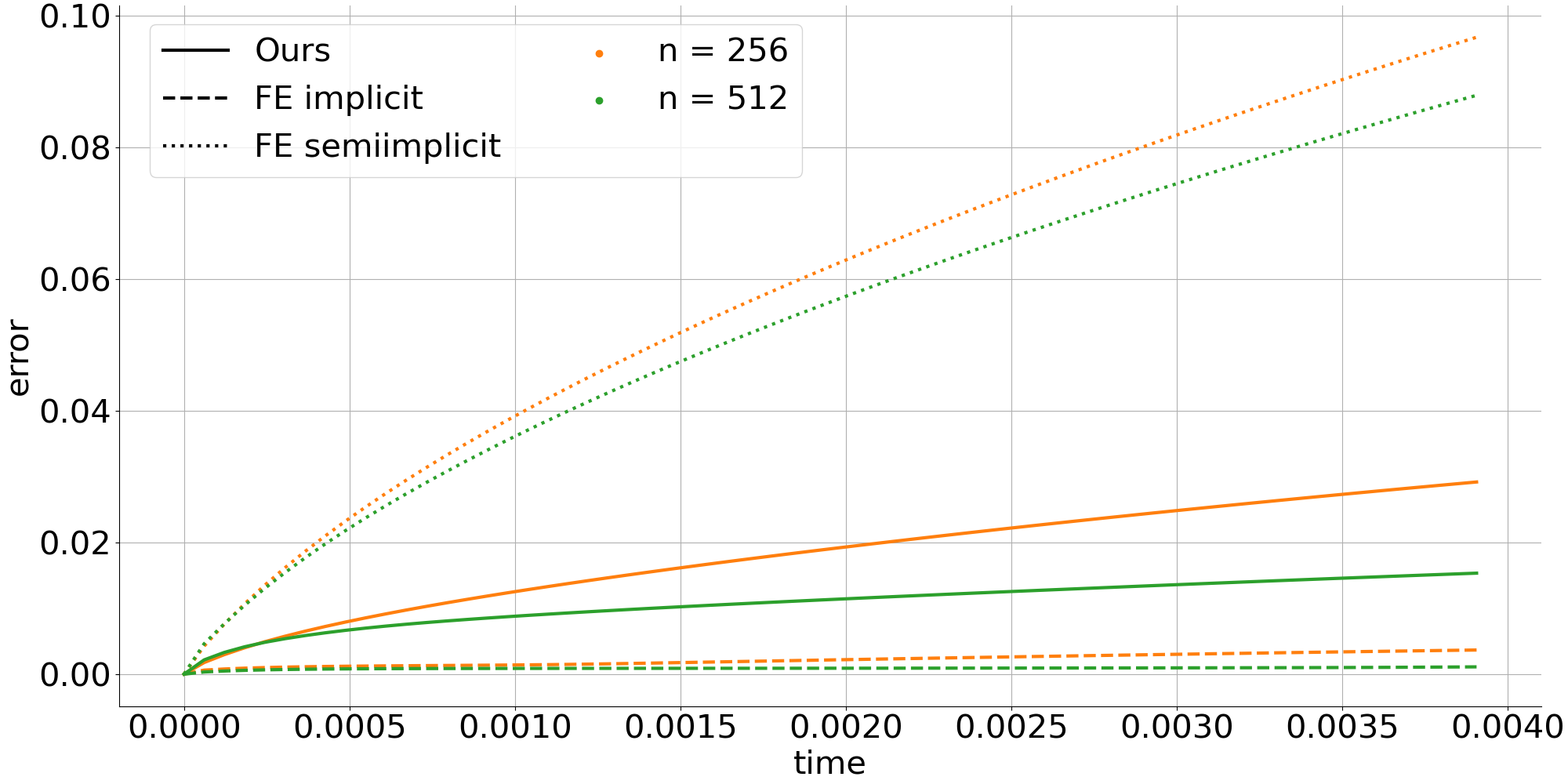}};
	
	\node[inner sep=0pt] at (-3.64, -2.3) {\includegraphics[width=.5\textwidth]{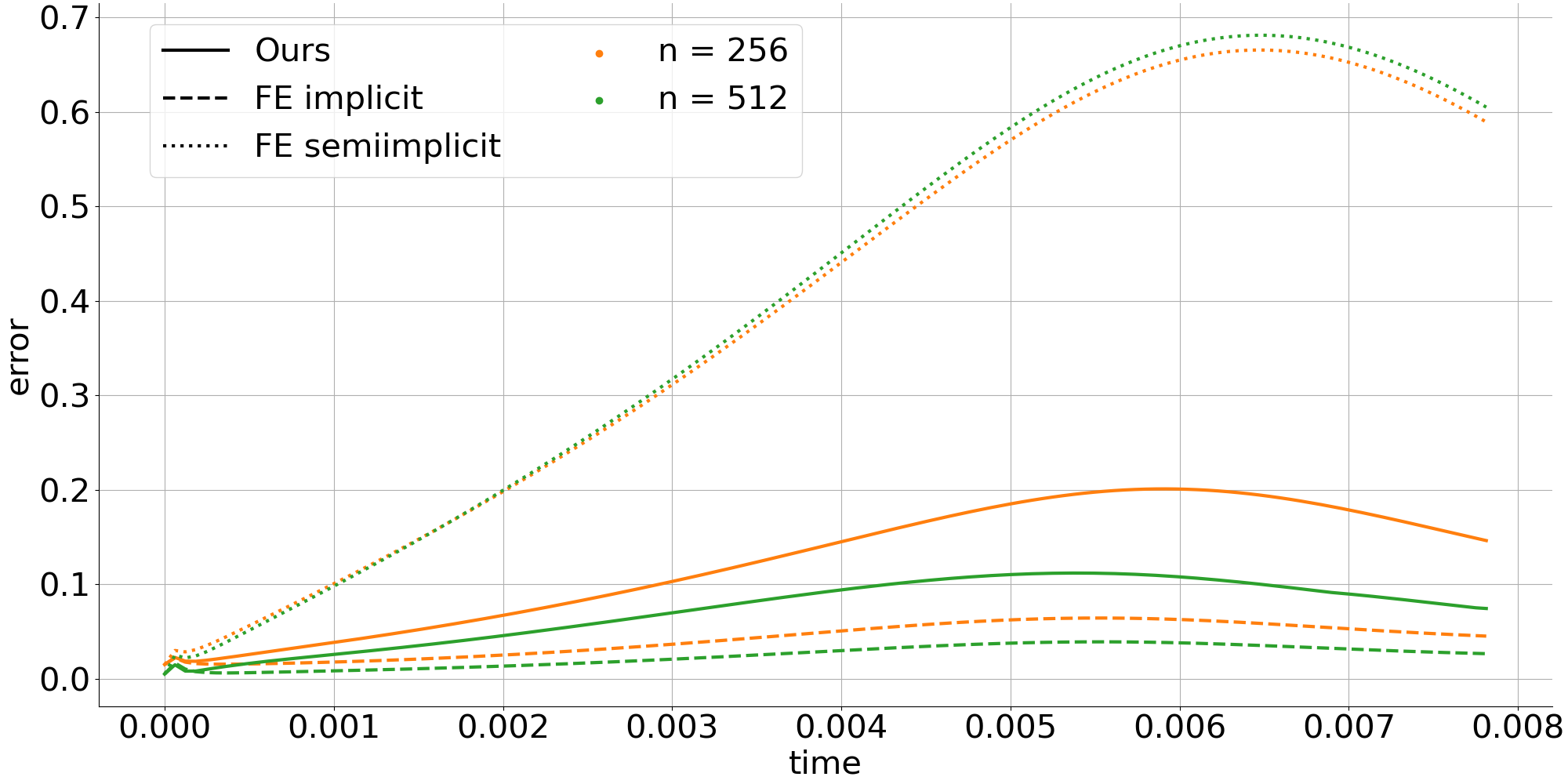}};
	\node[inner sep=0pt] at (4.24, -2.2) {\includegraphics[width=.34\textwidth, trim={3cm 2cm 3cm 2cm},clip]{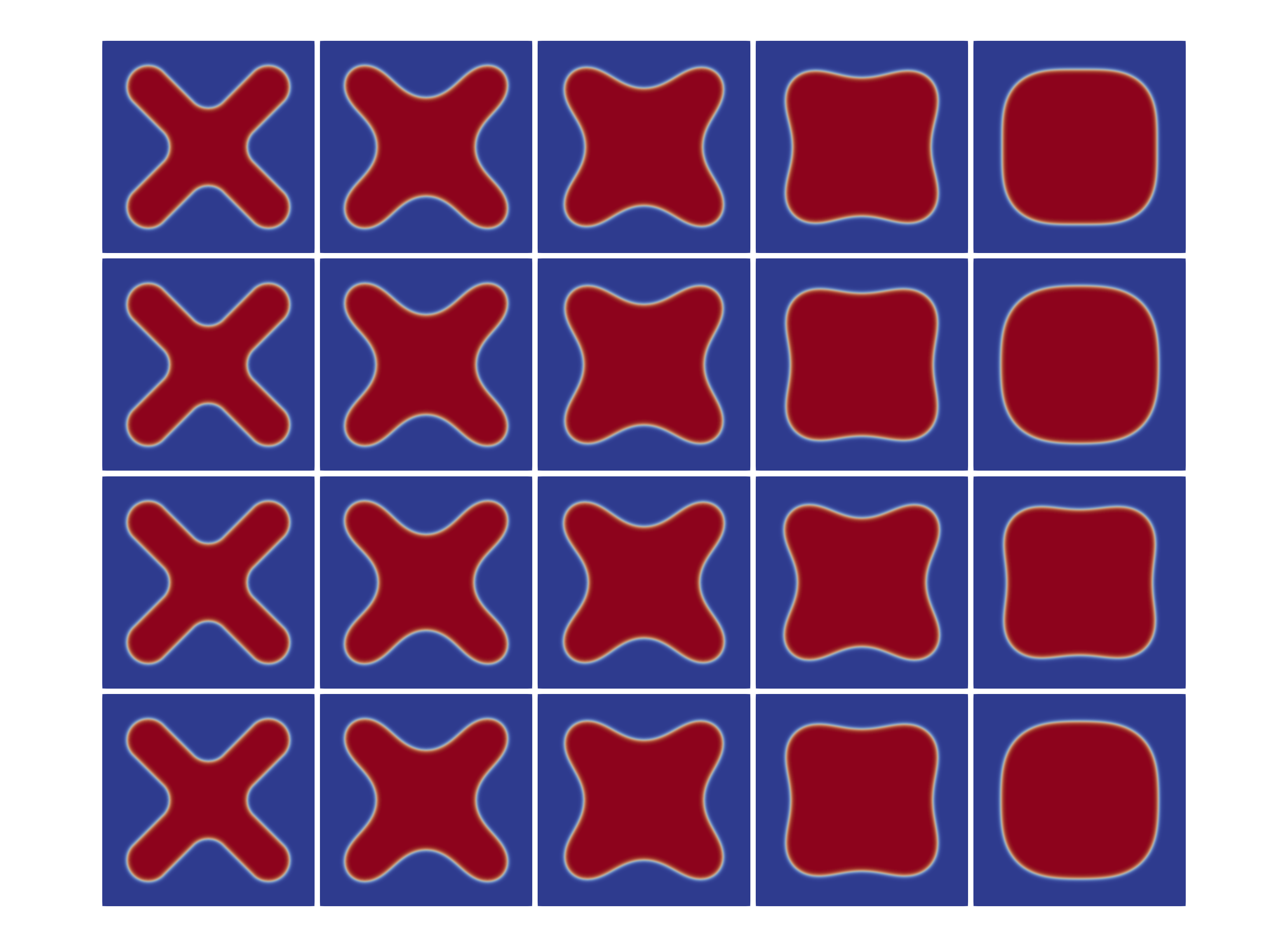}};
	\node[align=center] at (-3.74, 0.2) {Circles};
	\node[align=center] at (4.04, 0.2) {Rectangle};
	\node[align=center] at (0.4, -4.3) {Cross};
	
	\node[font=\tiny\linespread{0.9}\selectfont, align=center] at (1.25, -0.7) {Ours\\ $n=256$};
	\node[font=\tiny\linespread{0.9}\selectfont, align=center] at (1.25, -1.75) {FE\\ implicit\\ $n=256$};
	\node[font=\tiny\linespread{0.9}\selectfont, align=center] at (1.25, -2.7) {FE\\ semiimplicit\\ $n=256$};
	\node[font=\tiny\linespread{0.9}\selectfont, align=center] at (1.25, -3.6) {FE\\ implicit\\ $n=2048$};
	\node[font=\tiny\linespread{0.9}\selectfont, align=center] at (2.3, -4.23) {$t=0$};
	\node[font=\tiny\linespread{0.9}\selectfont, align=center] at (4.34, -4.21) {$t=2^{-8}$};
	\node[font=\tiny\linespread{0.9}\selectfont, align=center] at (6.14, -4.21) {$t=2^{-7}$};
	%		\node[font=\scriptsize\linespread{0.9}\selectfont, align=center] at (-2, 1.7) {Ours\\ $n=256$};
	%		\node[font=\scriptsize\linespread{0.9}\selectfont, align=center] at (-2, 0.5) {FE implicit\\ $n=256$};
	%		\node[font=\scriptsize\linespread{0.9}\selectfont, align=center] at (-2, -0.5) {FE semiimplicit\\ $n=256$};
	%		\node[font=\scriptsize\linespread{0.9}\selectfont, align=center] at (-2, -1.7) {FE implicit\\ $n=2048$};
	
	\end{tikzpicture}
	
	\caption{
		%\MR{$n=N$, $n_K$?}
		Convergence tests for Willmore flow with fixed \(\eps = 2^{-6}\) and \(\tilde \tau = 2^{-14}\) while increasing \(n\), with \(n_K = \tfrac{n}{8} + 1\). 
		As in \autoref{fig:ConvergenceTestMCF}, the line styles correspond to the different methods for approximating the mean curvature evolution and the colors to the varying resolution.
		In the top left, we plot the average $L^2$-error to the analytic solution for 30 circles with radii $r_i = 0.05\pi + \tfrac{0.15\pi i}{30}\,$, $i=0,\dots,29$ over time. 
		For comparison: for the averaged $L^2$ distance between the solution at time $0.004$ and at the initial time zero, one obtains$\tfrac{1}{30}\sum_{i=1}^{30} \|U_{R_W(r_i, 0.004)}- U_{R_W(r_i, 0)}\|_{L^2} \approx 0.414$. 
		In the top right, we plot the $L^2$-error of the evolution of a rectangle with sidelengths 0.4 and 0.2 compared to an implicit finite element solution with $n=2048$. In the bottom left, we plot the $L^2$-error of the evolution of a non-convex cross shape compared again to an implicit finite element solution with $n=2048$. In the bottom right, the computed evolution of the cross shape for the different schemes is displayed, for times $t=0$, $t=2^{-9}$, $t=2^{-8}$, $t=3\cdot 2^{-9}$, and $t=2^{-7}$.
	}
	\label{fig:ConvergenceTestWF}
	
\end{figure}
%ffffffffffffffffffffffffffffffffffffffffffffffffffffffff

One observes that the neural network-based MCF scheme performs noticeably better than the semi-implicit finite element scheme. 
This confirms the observations made by Bretin \etal~\cite{BrDeMa22} for a single resolution. 
The fully implicit FE scheme outperforms the network-based scheme, which is not that surprising due to its semi-implicit nature.
Besides the error caused by the semi-implicit approach, further error sources are related to the neural operator itself, such as limitations of the network architecture or training data set.
In \autoref{fig:ConvergenceTestMCF} one observes an increase of the error in time for increasing spatial resolution, which appears to be caused by a dominant time discretization error for fixed time step size $\tilde \tau$, which increases 
with the increasing spatial resolution. Indeed, for $n=1024$ we observed a further small 
increase of the error, but the error appeared to saturate (e.g., at about $\approx 0.01353$ for $t\approx0.0039$).

To validate the convergence of the neural network-based MCF scheme on other shapes than circles, where it was trained on, we also investigated the convergence for increasing $n$ on a nonconvex cross shape on the domain $\Omega = (-1,1)^2$. 
Since an analytic solution of mean curvature flow for this shape is not at hand, we compared the solutions of the different schemes to a solution of the nested finite element scheme \cite{FrRuWi11}  with $n=2048$, $\tilde \tau = 2^{-16}$. 
In \autoref{fig:ConvergenceTestMCFcross} on the left, the evolution of the error in time is depicted for the different schemes and resolutions. 
Here, $\eps$ and $\tilde \tau$ are chosen as before. 
Again, we see that the neural network-based scheme outperforms the semi-implicit finite element scheme, and in this example, even the implicit finite element scheme. 
That the observed error does not start at zero is due to the interpolation error between the computational resolutions ($n=128,256,512$) and the reference resolution ($n=2048$). 
The evolution of the cross shape is displayed on the right  in \autoref{fig:ConvergenceTestMCFcross}. We observe that in the time interval considered here, mean curvature flow varies only moderately, whereas the variation is significant for Willmore flow on the identical time interval, see \autoref{fig:ConvergenceTestWF}. 
%The interface length relevant to estimate the error in the interface position is  $6.968, 6.638, 6.329$ at times $t=0,\;2^{-8},\; 2^{-7}$ (given as the associated approximation of the Modica-Mortola energy).}

In summary, the network-based MCF scheme positions itself between the semi-implicit and the fully implicit finite element scheme in terms of accuracy.

%ffffffffffffffffffffffffffffffffffffffffffffffffffffffff

\begin{table}[htbp!]
	\centering
	
	\begin{tabular}{l|r|r}
		\toprule
		\textbf{2d cross shape}       & $n=256$ & $n=512$   \\
		\midrule
		Ours  &         122 sec     &   1020 sec           \\
		FE implicit   &      4970 sec   &   28622 sec      \\
		FE semiimplicit   &       4307 sec   &   22045 sec        \\
		
		\bottomrule
	\end{tabular}
	
	\caption{
		Computing time for 128 time steps (until time $t=0.0078125$) of Willmore flow of the different schemes at different resolutions.
	}
	\label{tab:PerformanceCross}
	
\end{table}

%ffffffffffffffffffffffffffffffffffffffffffffffffffffffff
In two dimensions, the evolution of a circle with initial radius $r$ under Willmore flow is given at time $\tau$ by a circle with radius $R_W(r,\tau) =  \sqrt[^4]{r^4 + 2\tau}$. 
In \autoref{fig:ConvergenceTestWF} in the top left, the error of our hybrid scheme based on \autoref{eq:discrete_willmore_flow_energy}, \autoref{eq:mcf_operator_discrete} is displayed, averaging over the same set of initial radii as in \autoref{fig:ConvergenceTestMCF}.
Again, the different line styles (solid, dashed, dotted) correspond to the different methods for solving the inner MCF problem (hybrid, finite element implicit and semi-implicit). 
As shown in \cite{FrRuWi11}, for $\eps=2h$ ($h=\tfrac{1}{n}$), the nested Willmore scheme suffers from numerical instabilities.
Hence, we only consider the error evolution for $n=256,\, 512$ and for fixed $\eps = 2^{-6}$. 

To test our scheme not only on circles, on which the networks were trained, we applied the hybrid scheme to initial data given as the phase field approximation of a rectangle sized $0.4\times 0.2$.
Since there is no known analytic solution for Willmore flow in this case, we compare the results to a solution of the nested finite element scheme by Franken \etal \cite{FrRuWi11} on a much finer resolution ($n=2048$). 
The corresponding error evolution is displayed in \autoref{fig:ConvergenceTestWF} on the top right.
In addition, we tested our scheme for the cross-shaped initial condition, whose MCF evolution is already examined in \autoref{fig:ConvergenceTestMCFcross}. Again, we compared to the solution of the nested scheme with $n=2048$, cf.~\autoref{fig:ConvergenceTestWF} on the bottom left. 

In all three experiments, we observe that our hybrid approach performs noticeably better than the nested scheme from \cite{FrRuWi11} using a semi-implicit finite element approach for the inner MCF problem, while the fully implicit nested finite element scheme remains the most accurate. This mirrors the observations for MCF in the circle experiment.

Regarding computing time, our hybrid scheme clearly outperforms both finite element schemes. \autoref{tab:PerformanceCross} displays the time required to compute the $128$ time steps for the evolution of the cross shape, with an online computing time of our hybrid scheme more than $20$ times faster. Here, the computing time of the offline training of the neural operator has not been taken into account. However, once the model is trained, our experiments indicate that the reduction in computing time appears to be independent of different choices of the initial data.

In \autoref{fig:ConvergenceTestWF} on the bottom right, we show the evolution of the cross shape under Willmore flow. Compared to the corresponding evolution under MCF, we observe a significantly faster shape variation in the case of Willmore flow, 
evolving it into an already convex shape at time $t = 2^{-7}$. %, after $128$ time steps. 
It is clearly visible that the finite element scheme relying on a semi-implicit time discretization of the inner MCF problem leads to a poorer approximation of Willmore flow, while the results of our scheme are qualitatively indistinguishable from those computed with $n=2048$ already for $n=256$. 
In particular, for applications where one is primarily interested in the qualitative effects of Willmore flow, one can effectively achieve satisfying results 
already on coarser resolutions, as it will be exploited in the next section.

\paragraph{Remarks on the implementation and the data processing pipeline.}
We end this section with details on the implementation and data processing pipeline to facilitate reproducibility.
As for our hybrid scheme (see \autoref{sec:discretization}), we have implemented the finite element-based approaches in Python using standard libraries.
We implemented the nested scheme using Newton's method as described in \cite{FrRuWi11} with direct linear solver PARDISO \cite{ScGaFi01}.
All experiments using the finite element-based approaches were run on a workstation with two 32-core AMD EPYC 7601 processors with 1TB RAM. 
All experiments using our hybrid approach were run on a workstation with an NVIDIA A100 GPU with 40GB of memory and two 24-core AMD EPYC 7402 processors with 256GB RAM using double-precision floating-point arithmetic on the GPU.
The training of the neural MCF operator benefits from additional memory as it allows to cache more of the training data.
However, the training is possible with less memory at the moderate expense of speed. 
Running our scheme, once the MCF operator is trained, requires only a small fraction of the available GPU and CPU memory.
For the largest 3D experiment (cf.~\autoref{fig:armadillo}) $6.6$ GB GPU cache is used for $n=128$.
The phase fields for the armadillo and the rocker-arm were generated as follows:
First, using the tool $\texttt{mesh\_to\_voxels}$ from \cite{Kl21} signed distances were generated on a $n^3$ grid given a triangular mesh.
Then, the optimal phase field profile, as described in \autoref{sec:method}, was concatenated with the signed distances to obtain the initial phase field approximation.
To generate the images in \autoref{fig:3dInpainting} and in \autoref{fig:armadillo}, the zero-levelset of the phase field was extracted using the $\texttt{Contour}$ function from Paraview \cite{AhGeLa05}. 
Finally, all results in 3D were rendered in Blender \cite{Bl18}.

%ffffffffffffffffffffffffffffffffffffffffffffffffffffffff
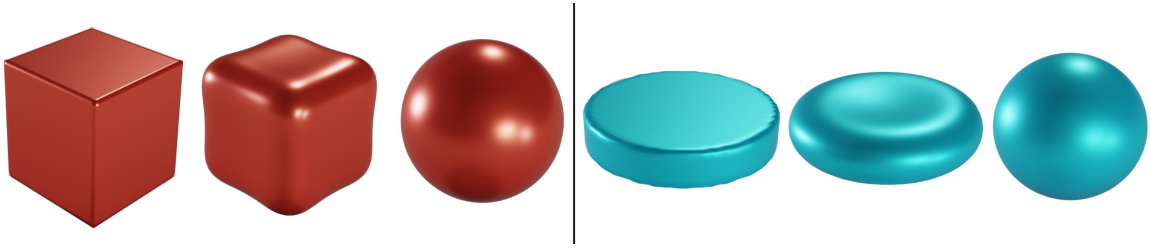
\begin{figure}[t]
	\centering
	\begin{tikzpicture}
	\node at (0,-1.4) {\includegraphics[width=\textwidth, trim={0 5cm 0cm 0cm},clip]{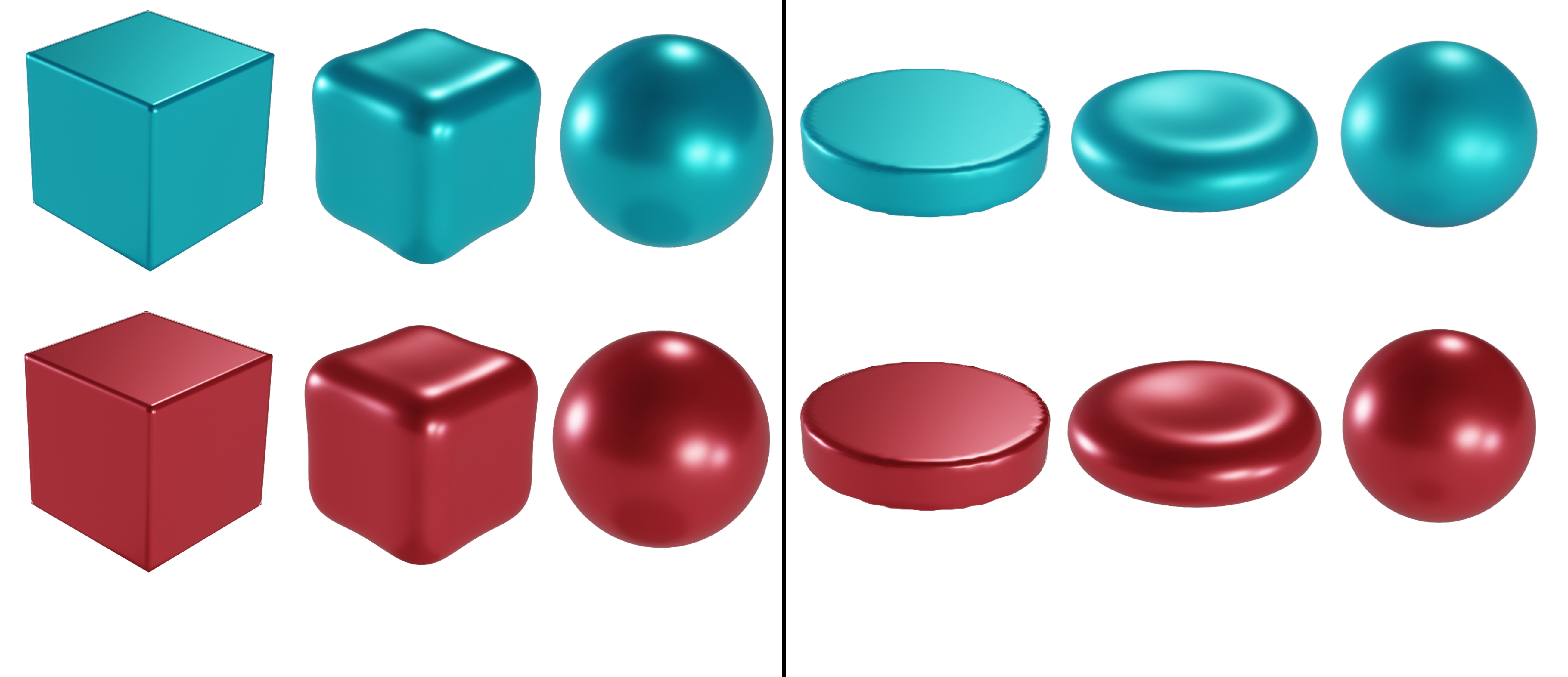}};
	\node[font=\small\linespread{0.9}\selectfont, align=center, rotate=90] at (-7.7,0.1) {Ours};
	\node[font=\small\linespread{0.9}\selectfont, align=center, rotate=90] at (-7.7,-2.9) {FE implicit};
	\node[font=\tiny\linespread{0.9}\selectfont, align=center] at (-6.4,1.5) {$w_{\tilde\tau}^\eps[u_{\scaleto{\text{NN}}{2pt}}^{\scaleto{0}{2pt}}] = 132.65$};
	\node[font=\tiny\linespread{0.9}\selectfont, align=center] at (-3.6,1.5) {$w_{\tilde\tau}^\eps[u_{\scaleto{\text{NN}}{2pt}}^{\scaleto{1}{2pt}}] = 52.26$};
	\node[font=\tiny\linespread{0.9}\selectfont, align=center] at (-1.1,1.5) {$w_{\tilde\tau}^\eps[u_{\scaleto{\text{NN}}{2pt}}^{\scaleto{100}{2pt}}] = 25.78$};
	\node[font=\tiny\linespread{0.9}\selectfont, align=center] at (6.5,1.5) {$w_{\tilde\tau}^\eps[u_{\scaleto{\text{NN}}{2pt}}^{\scaleto{100}{2pt}}] = 25.59$};
	\node[font=\tiny\linespread{0.9}\selectfont, align=center] at (4.,1.5) {$w_{\tilde\tau}^\eps[u_{\scaleto{\text{NN}}{2pt}}^{\scaleto{1}{2pt}}] = 50.49$};
	\node[font=\tiny\linespread{0.9}\selectfont, align=center] at (1.3,1.5) {$w_{\tilde\tau}^\eps[u_{\scaleto{\text{NN}}{2pt}}^{\scaleto{0}{2pt}}] = 87.15$};

	%		\node[font=\tiny\linespread{0.9}\selectfont, align=center] at (-2.9,-1.35) {$\|u_{\scaleto{\text{NN}}{2pt}}^{\scaleto{1}{2pt}} - u_{\scaleto{\text{FE}}{2pt}}^{\scaleto{1}{2pt}}\|_{\scaleto{L^2}{4pt}} = 7.66\mathrm{e}{-3}$};
	\node[font=\tiny\linespread{0.9}\selectfont, align=center] at (-3.7,-1.35) {$\|u_{\scaleto{\text{NN}}{2pt}}^{\scaleto{1}{2pt}} - u_{\scaleto{\text{FE}}{2pt}}^{\scaleto{1}{2pt}}\| = 7.66\mathrm{e}{-3}$};
	\node[font=\tiny\linespread{0.9}\selectfont, align=center] at (-1.3,-1.35) {$\|u_{\scaleto{\text{NN}}{2pt}}^{\scaleto{100}{2pt}} - u_{\scaleto{\text{FE}}{2pt}}^{\scaleto{100}{2pt}}\| = 7.69\mathrm{e}{-3}$};
	\node[font=\tiny\linespread{0.9}\selectfont, align=center] at (3.8,-1.35) {$\|u_{\scaleto{\text{NN}}{2pt}}^{\scaleto{1}{2pt}} - u_{\scaleto{\text{FE}}{2pt}}^{\scaleto{1}{2pt}}\| = 4.08\mathrm{e}{-3}$};
	\node[font=\tiny\linespread{0.9}\selectfont, align=center] at (6.3,-1.35) {$\|u_{\scaleto{\text{NN}}{2pt}}^{\scaleto{100}{2pt}} - u_{\scaleto{\text{FE}}{2pt}}^{\scaleto{100}{2pt}}\| = 27.01\mathrm{e}{-3}$};
	
	%
	%		\node[font=\tiny\linespread{0.9}\selectfont, align=center] at (-2.9,-3.75) {$w_{\tilde\tau}^\eps[u_{\scaleto{\text{FE}}{2pt}}^{\scaleto{1}{2pt}}] = 53.6$};
	%		\node[font=\tiny\linespread{0.9}\selectfont, align=center] at (-0.9,-3.75) {$w_{\tilde\tau}^\eps[u_{\scaleto{\text{FE}}{2pt}}^{\scaleto{100}{2pt}}] = 25.03$};
	%		\node[font=\tiny\linespread{0.9}\selectfont, align=center] at (5.3,-3.75) {$w_{\tilde\tau}^\eps[u_{\scaleto{\text{FE}}{2pt}}^{\scaleto{100}{2pt}}] = 25.03$};
	%		\node[font=\tiny\linespread{0.9}\selectfont, align=center] at (3.2,-3.75) {$w_{\tilde\tau}^\eps[u_{\scaleto{\text{FE}}{2pt}}^{\scaleto{1}{2pt}}] = 50.3$};
	
	\end{tikzpicture}
	
	\caption{
		Evolution by Willmore flow for a cube surface and a thick disk surface with spatial resolution $n = 64$, interface parameter $\eps=2^{-5}$, inner stepsize $\tilde{\tau} = 2^{-12}$, and Willmore flow stepsize $\tau = 2^{-18}$. Top row: Our hybrid scheme with kernel size $n_K =17$. Bottom row: nested implicit finite element scheme from \cite{FrRuWi11}.
		Results are shown at time \(0,1\), and time \(100\), which reproduce concave surface patches as common effects for Willmore. The approximate phase field Willmore energy $w_{\tilde \tau}^\eps[u] = \tfrac\eps{2\tilde\tau^2}\|v_{\tilde\tau}[u] - u\|^2$, and the $L^2$-error between the results of the two schemes are also displayed. For comparison: The exact Willmore energy of a sphere is $8\pi \approx 25.13\,$.
	}
	\label{fig:DiskAndCube}
\end{figure}
%ffffffffffffffffffffffffffffffffffffffffffffffffffffffff

%%%%%%%%%%%%%%%%%%%%%%%%%%%%%%%%
%%%%%%%%%%%%%%%%%%%%%%%%%%%%%%%%
%%%%%%%%%%%%%%%%%%%%%%%%%%%%%%%%
\section{Applications in image and geometry processing}
\label{sec:applications}
Now that we have introduced and studied our proposed scheme, we will briefly discuss its use for two closely related applications: surface fairing and surface restoration.

In the first of these applications, \emph{surface fairing}, the goal is to create a visually smooth and seamless surface starting from an input surface that is usually noisy or otherwise corrupted.
The problem has been studied in the computer graphics and vision communities and curvature flows have established themselves as a classical tool \cite{DeMeSc99}.
Willmore flow is particularly interesting as a basic model for this application \cite{BoSc05,GrAu20,SoChDi21} since it avoids singularities that could arise with mean curvature flow \cite{CrPiSc13} 
and preserves $C^1$ boundary conditions.
Certainly, other flows decreasing the Willmore energy, \eg an $H^2$-flow \cite{Sc17}, would also be suitable. 
In \autoref{fig:DiskAndCube}, we consider the Willmore flow of surfaces in $\R^3$ for different initial surfaces and compare to the 
scheme from \cite{FrRuWi11}. Additionally, we see that the approximate Willmore energy of the final shapes $$w_{\tilde \tau}^\eps[u] \coloneqq \tfrac\eps{2\tilde\tau^2}\|v_{\tilde\tau}[u] - u\|^2$$, with $v_{\tilde\tau}[u]$ given by \eqref{eq:implicitMCF}, discretized by finite elements, is close to $8\pi\approx 25.13$, the exact Willmore energy of a round sphere.
As the evolution under Willmore flow in 3D is quite fast, we choose a small outer step size $\tau$ to prevent the shape losing too many details too quickly.
It is in fact usually smaller than the inner MCF step size $\tilde \tau$. 
To show stability for large step sizes, we consider the evolution of the armadillo, often referenced in computer graphics, in \autoref{fig:armadillo}, were we start with the very small outer step size $\tau = 2^{-22}$, but double it after every time step, leading to a almost sphere-like shape in 16 time steps.

In \autoref{fig:rockerarm}, the Willmore flow of a rocker-arm is displayed, showing the evolution towards a Clifford torus.
In \cite{MaNe14}, it was shown that such a torus is a minimizer of the Willmore energy among immersed tori in $\R^3$ 
with Willmore energy $4\pi^2\approx 39.48$. 
Our final shape, after 20 time steps, has an approximate Willmore energy 
$w_{\tilde\tau}^\eps[u_{\scaleto{\text{NN}}{4pt}}^{\scaleto{20}{4pt}}] \approx 41.46$, close to $4\pi^2$. 
To visually compare the final shape to a Clifford torus, we fitted a phase field $u_{\scaleto{\text{Clifford}}{4pt}}$ representing a Clifford torus to the phase field of the final shape $u_{\scaleto{\text{NN}}{4pt}}^{\scaleto{20}{4pt}}$. 
A phase field representation of the Clifford torus is obtained by concatenating the optimal Modica--Mortola profile and a signed distance function of the 
torus. The free parameters for this fitting are the center coordinates and a SO(3) rotation, and the outer radius. 
Indeed, the final shape is close to a Clifford torus, with $\|u_{\scaleto{\text{NN}}{4pt}}^{\scaleto{20}{4pt}} - u_{\scaleto{\text{Clifford}}{4pt}}\|_{L^2} \approx 3.978\cdot10^{-2}\,$. The phase field representing a Clifford torus has an approximate Willmore energy 
$w_{\tilde\tau}^\eps[u_{\scaleto{\text{Clifford}}{4pt}}] \approx 41.11\,$.

Concerning computational cost, we show the results comparing our hybrid scheme to the original finite element-based nested Willmore approach by Franken \etal \cite{FrRuWi11} in \autoref{table:Performance}.

%ffffffffffffffffffffffffffffffffffffffffffffffffffffffff
\begin{figure}[t]
	
	\includegraphics[width=\textwidth]{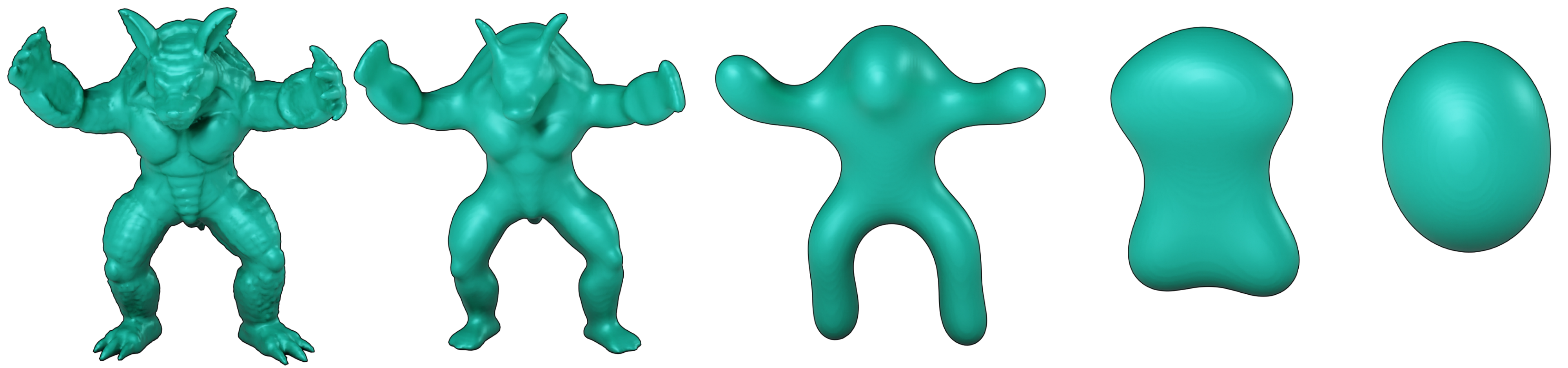}
	
	\caption{
		Evolution of the armadillo shape under Willmore flow for $n = 128$ and $n_K =17$, $\eps=2^{-6}$, and $\tilde{\tau} = 2^{-14}$. 
		The initial outer step size is $\tau = 2^{-22}$. After every time step, the outer step size is doubled.
		Results are shown at time steps \(0,1,8,11\), and \(16\), respectively. 
	}
	\label{fig:armadillo}
\end{figure}
%ffffffffffffffffffffffffffffffffffffffffffffffffffffffff

%ffffffffffffffffffffffffffffffffffffffffffffffffffffffff
\begin{figure}[t]
	\centering
	\begin{tikzpicture}
	\node at (0,0) {\includegraphics[width=\textwidth]{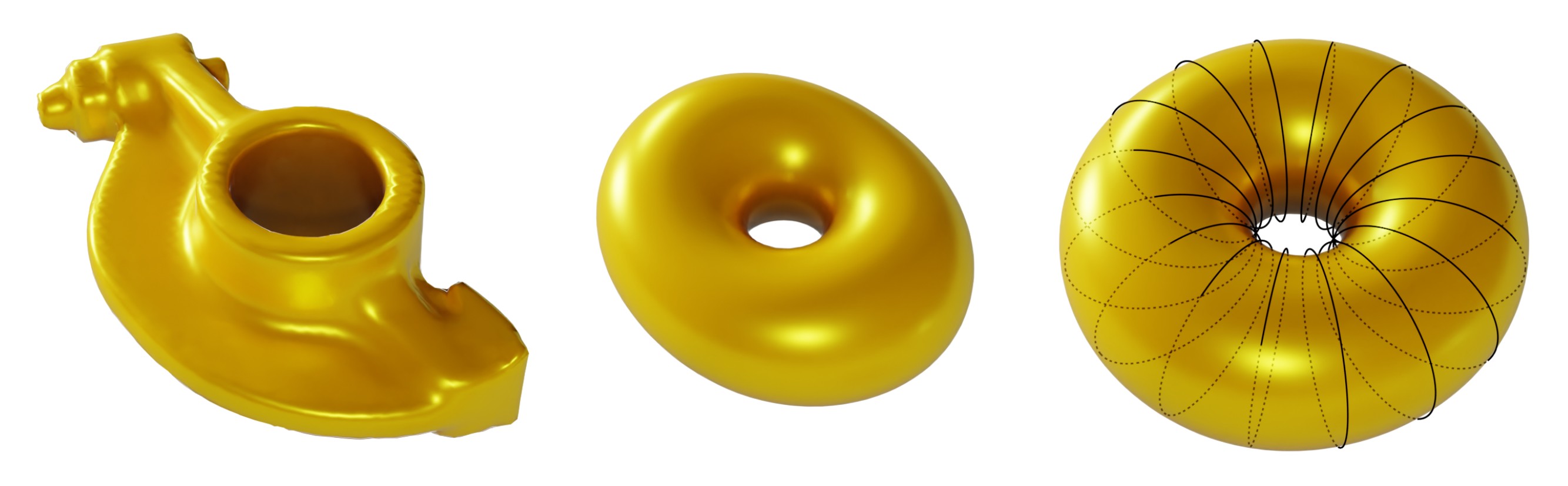}};
	\node[font=\tiny\linespread{0.9}\selectfont, align=center] at (-4.9,2.15) {$w_{\tilde\tau}^\eps[u_{\scaleto{\text{NN}}{2pt}}^{\scaleto{0}{2pt}}] = 276.12$};
	\node[font=\tiny\linespread{0.9}\selectfont, align=center] at (0,2.15) {$w_{\tilde\tau}^\eps[u_{\scaleto{\text{NN}}{2pt}}^{\scaleto{1}{2pt}}] = 42.08$};
	\node[font=\tiny\linespread{0.9}\selectfont, align=center] at (4.9,2.15) {$w_{\tilde\tau}^\eps[u_{\scaleto{\text{NN}}{2pt}}^{\scaleto{20}{2pt}}] = 41.46$};
	\node[font=\tiny\linespread{0.9}\selectfont, align=center] at (4.9,-2.25) {$\|u_{\scaleto{\text{NN}}{2pt}}^{\scaleto{20}{2pt}} - u_{\scaleto{\text{Clifford}}{2pt}}\| = 3.978\cdot10^{-2}$};
	\end{tikzpicture}
	
	\caption{
		Evolution of a rocker arm shape under Willmore flow.
		The spatial resolution is $n = 64$, the kernel size is $n_K =17$, 
		the interface parameter $\eps=2^{-5}$, the inner MCF step size $\tilde{\tau} = 2^{-12}$, and the Willmore flow step size $\tau = 2^{-12}$. 
		We show the discrete timesteps \(0,1\), and \(20\) leading to an approximation of the Clifford torus.
		For the last time step, we additionally draw minor circumcircles of the best-approximating Clifford torus. 		
	}
	\label{fig:rockerarm}
\end{figure}
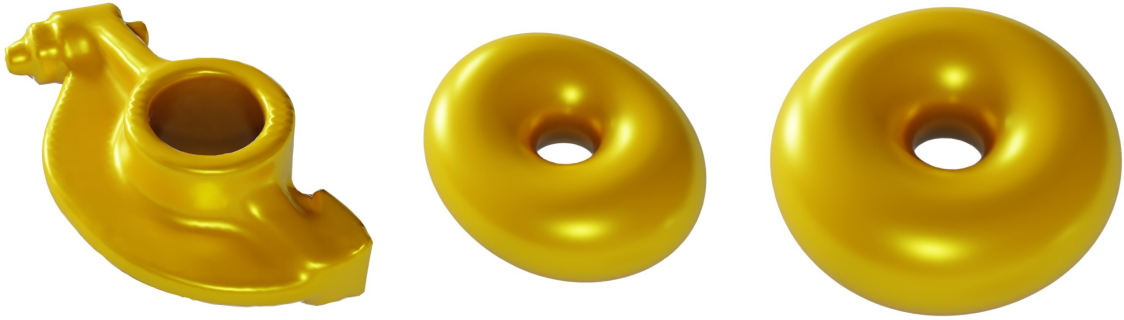
%ffffffffffffffffffffffffffffffffffffffffffffffffffffffff

\begin{table}[h!]
	\begin{subfigure}[c]{.49\linewidth}
		\begin{tabular}{l|r}
			\toprule
			\textbf{Method}       & \textbf{Armadillo ($n=64$)}  \\
			%		& \textbf{Armadillo} & \textbf{Cross} & \textbf{Ball} \\
			\midrule
			nested FEM \cite{FrRuWi11}   &         14589 sec                 \\
			%		 &                    &               &              \\
			Ours   &       807 sec               \\
			%		    &                    &               &              \\
			%		Ours (float, A100)    &                &              &                    &               &              \\
			\bottomrule
		\end{tabular}
	\end{subfigure}
	\hfill
	\begin{subfigure}[c]{.49\linewidth}
		\includegraphics[width=\textwidth]{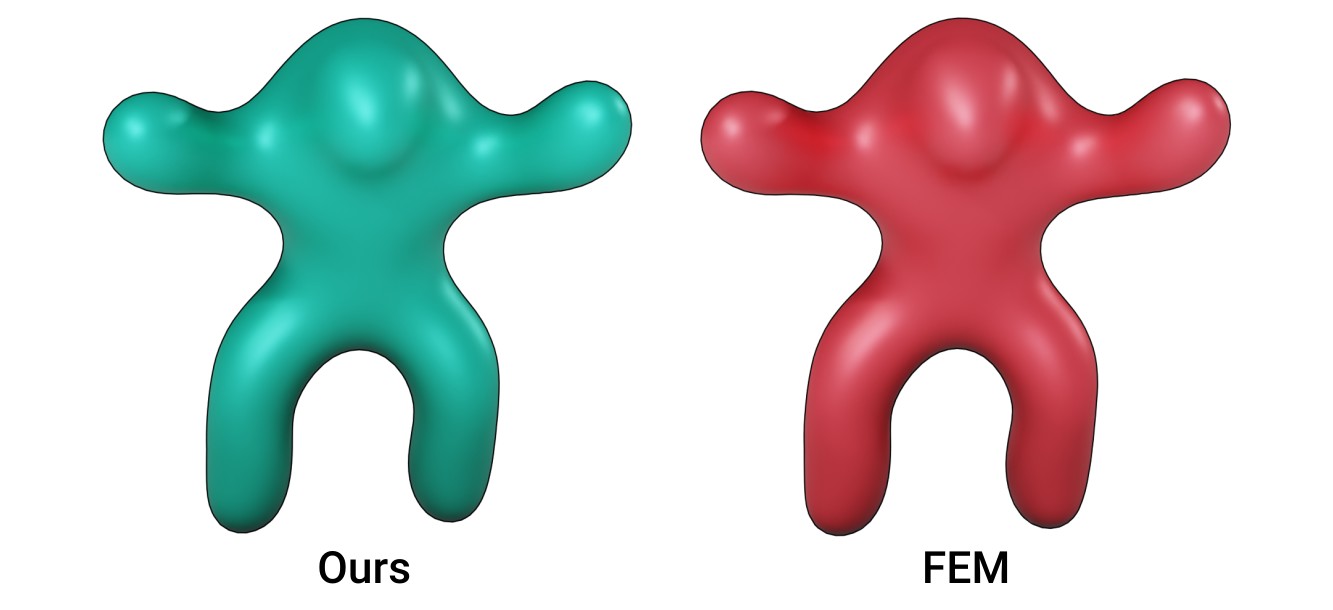}
	\end{subfigure}
	
	\caption{
		Performance comparison of our hybrid method to the nested finite element scheme by Franken \etal \cite{FrRuWi11} on the example from \autoref{fig:armadillo}:
		The results after one step of Willmore flow starting from the shape on the far-left of \autoref{fig:armadillo} (right), computing time (left). The parameters are $n=64$, $\eps=2^{-5}$, $\tau=2^{-18}$, and $\tilde\tau=2^{-12}\,$.
	}
	\label{table:Performance}
	
\end{table}

The second application we consider, \emph{surface inpainting}, is a fundamental topic in geometry and image processing, where one aims to restore corrupted or destroyed parts of an image or a surface.
We investigate in this section the use of our hybrid approach to tackle this problem.
In a first step, one replaces the corrupted or missing part by an ansatz geometry, whose primary purpose is to prescribe the desired topology.
Following the edge restoration approach by Nitzberg \etal \cite{NiMuSh93}, one considers the energy \eqref{eq:discrete_willmore_flow_energy} on the full domain \(\Omega\), but only takes into account degrees of freedom in a part \(D\subset\Omega\), where the image or surface 
is corrupted.  
Then, Willmore flow under this constraint leads to smooth reconstructions of the corrupted area while preserving $C^1$ boundary conditions on \(\partial D \cap \Omega\). 
We illustrated two examples in two dimensions in \autoref{fig:2dInpainting}, where the reconstruction region \(D\) is outlined in green. 
In \autoref{fig:3dInpainting}, we carried out similar experiments in three dimensions. 
A particularity of Willmore flow in two dimensions is that the reconstruction of the circle in \autoref{fig:2dInpainting} takes the enormous number of \(2000\) timesteps of size $2h$.
This is due to the competition of the convex and concave parts in the reconstruction area: positive curvature on the sides pushes the surface outwards, while the middle part with negative curvature tends to move inwards. 
Only because the curvature of the outer part dominates slightly, the surface moves slowly outwards. 
In contrast, the reconstruction of the ball in \autoref{fig:3dInpainting} does not have the same slow behavior as the circle, because the three-dimensional Willmore energy is scale-invariant.
%Since this behavior is inherent to Willmore flow it does not influence the computational cost im comparison to an finite element-based scheme, which can be seen in \autoref{table:Performance} as before.  

%ffffffffffffffffffffffffffffffffffffffffffffffffffffffff
\begin{figure}[htbp!]
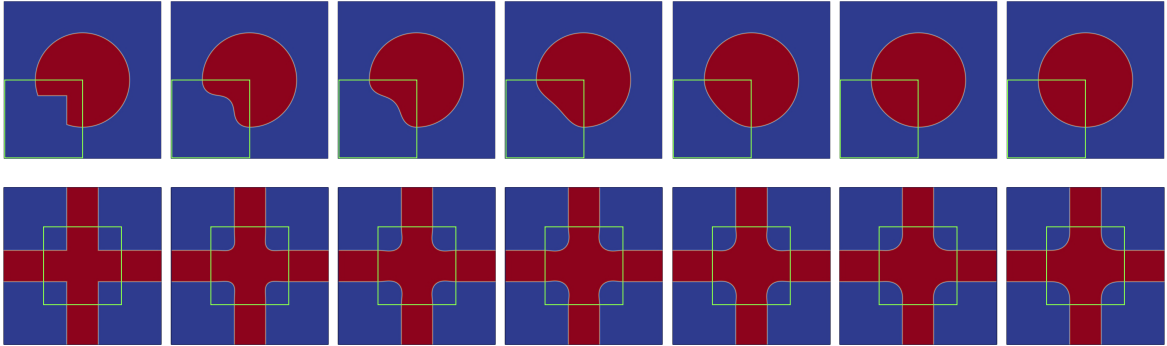

	\includegraphics[width=\textwidth]{images/circle_inpainting_withFrame_compressed}
	\includegraphics[width=\textwidth]{images/cross_inpainting_withFrame_compressed}
	\caption{
		Application of our hybrid scheme to image restoration in two dimensions.
		In both examples, the spatial resolution is $n = 1024$, the kernel size is $n_K =17$, and the interface parameter $\eps=2^{-8}$ and the reconstruction region \(D\) is outlined in green.
		In the top row, a disk with a cut-out corner is considered as the initial image and we use $\tilde{\tau} = 2^{-14}$, $\tau = 2^{-7}$.
		The restoration evolution is shown at times \(0, 1, 100, 500, 1000, 1500\), and \(2000\).
		In the bottom row, a cross shape is given as the initial image, where the central part is smoothed under discrete Willmore flow with stepsizes $\tilde{\tau} = 2^{-14}$ and $\tau = 2^{-14}$.
		We show the flow at times \(0, 1, 5, 10, 20, 35\), and \(50\).
	}
	\label{fig:2dInpainting}
\end{figure}
%ffffffffffffffffffffffffffffffffffffffffffffffffffffffff

%ffffffffffffffffffffffffffffffffffffffffffffffffffffffff
\begin{figure}[htbp!]
	\includegraphics[width=\textwidth]{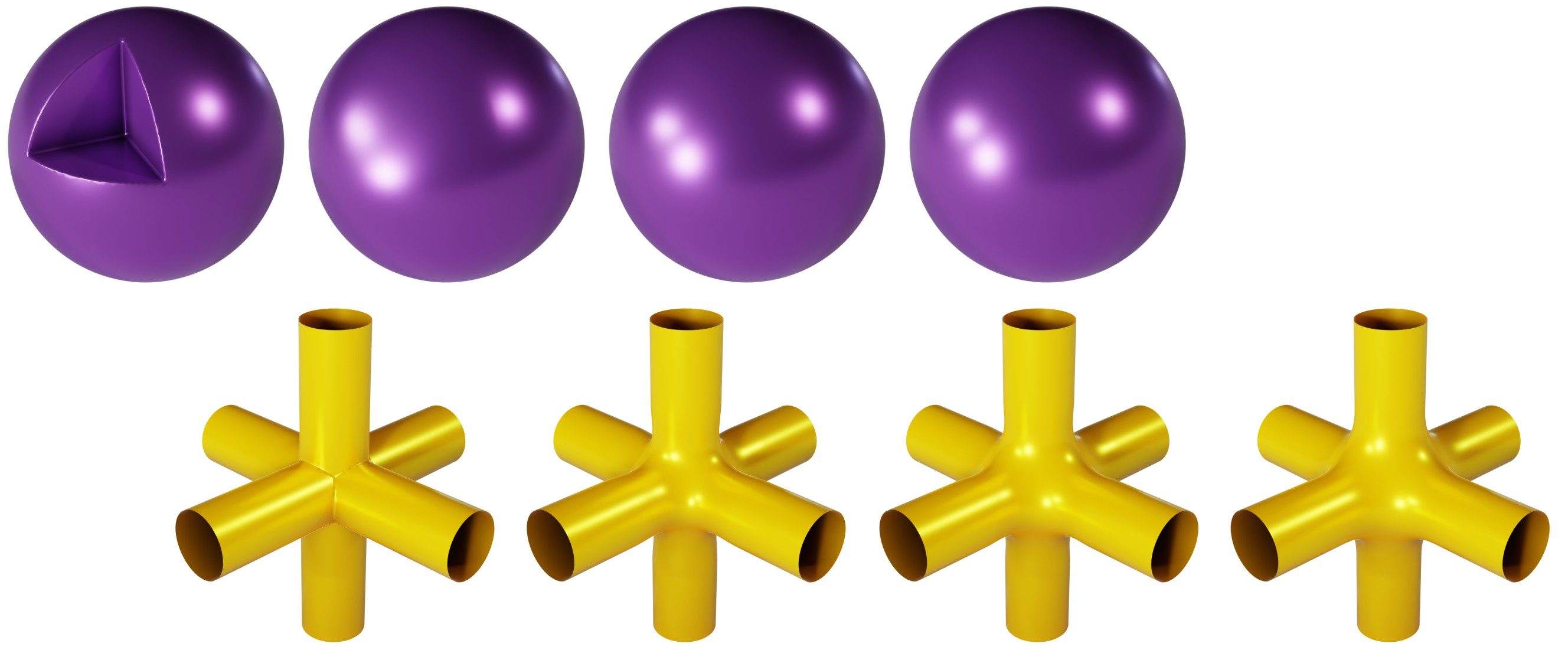}
	\caption{
		Application of our hybrid scheme to surface restoration in three dimensions. 
		In both examples, the spatial resolution was $n = 128$, the kernel size was $n_K =17$, the interface parameter was $\eps=2^{-6}$, and the timestep sizes were
		$\tilde{\tau} = 2^{-14}$ and $\tau = 2^{-21}$.
		In the top row, the flow at timesteps \(0, 1, 3\), and \(5\) recovering a sphere with an initially cut-out corner is shown.
		In the bottom row, the flow at timesteps \(0, 1, 4\), and \(8\) leads to a smooth blending of six tubes.
	}
	\label{fig:3dInpainting}
\end{figure}
%ffffffffffffffffffffffffffffffffffffffffffffffffffffffff

%%%%%%%%%%%%%%%%%%%%%%%%%%%%%%%%
%%%%%%%%%%%%%%%%%%%%%%%%%%%%%%%%
%%%%%%%%%%%%%%%%%%%%%%%%%%%%%%%%
\section{Conclusions}
\label{sec:conclusions}
In this paper, we presented a new hybrid scheme for Willmore flow in a phase field formulation, which combines a minimizing movement ansatz for the flow proposed by Franken \etal \cite{FrRuWi11} with a neural operator approach to compute an approximation of the mean curvature following Bretin \etal \cite{BrDeMa22}.
For a fixed timestep size and scale parameter of the phase field ansatz, the hybrid scheme shows error reduction for increasing spatial resolution and stencil size of the network kernel.
The new scheme comes with significantly reduced computing times. 
The resulting discrete Willmore flow properly reflects the qualitative behavior of the continuous flow and is, for instance, applicable to the restoration of 2D images and 3D surfaces. 

The results encourage the use of neural networks when simulating geometric flows. 
A future challenge would be to directly learn a solution operator for Willmore flow. 
In \cite{GrHe08}, Grzhibovskis and Heintz described a convolution thresholding scheme for Willmore flow. 
Hence, designing a neural network to directly learn Willmore flow does not seem to be out of reach. 
However, creating proper training data is more subtle. 
The striking observation in \cite{BrDeMa22} is that the evolution of spheres under mean curvature flow is sufficient for the approximation of mean curvature flow for a wide range of initial data.
For Willmore flow, surely a significantly richer set of training data is required.

\section*{Acknowledgments}
We thank Angelo Kitio for helping with the initial re-implementation of the network-based scheme for mean curvature flow by Bretin \etal \cite{BrDeMa22}.

\section*{Funding}
This work was supported by the Deutsche Forschungsgemeinschaft (DFG, German Research Foundation) via project 211504053 -- Collaborative Research Center 1060 and via Germany’s Excellence Strategy project 390685813 -- Hausdorff Center for Mathematics.
Furthermore, this project has received funding from the European Union’s Horizon 2020 research and innovation program under the Marie Skłodowska-Curie grant agreement No 101034255.

%------
% Insert the bibliography.
%------
\bibliographystyle{siam}
\bibliography{RuSaSm24}

\end{document}